\documentclass[onefignum,onetabnum]{siamart171218}

\usepackage{amssymb}
\usepackage{chemarrow}
\usepackage[all]{xy}
\usepackage{enumerate}

\usepackage{lipsum}
\usepackage{amsfonts}
\usepackage{graphicx}
\usepackage{epstopdf}
\usepackage{algorithmic}
\ifpdf
  \DeclareGraphicsExtensions{.eps,.pdf,.png,.jpg}
\else
  \DeclareGraphicsExtensions{.eps}
\fi

\newsiamremark{remark}{Remark}
\newsiamremark{hypothesis}{Hypothesis}
\crefname{hypothesis}{Hypothesis}{Hypotheses}
\newsiamthm{claim}{Claim}

\headers{Decoupling of Mixed Methods Based on Helmholtz Decompositions}{L. Chen and X. Huang}

\title{Decoupling of Mixed Methods Based on Generalized Helmholtz Decompositions\thanks{Submitted to the editors \today.
\funding{The first author was supported by  NSF Grant DMS-1418934 and in part by the Sea PolyProject of Beijing Overseas Talents. The second author was supported by the NSFC Projects 11771338 and 11671304, Zhejiang Provincial
Natural Science Foundation of China Projects LY17A010010 and LY15A010015, and Wenzhou Science and Technology Plan Project G20160019.}}}

\author{Long Chen\thanks{Beijing Institute for Scientific and Engineering Computing. Beijing University of Technology. Beijing 100124, China.
Department of Mathematics, University of California at Irvine, Irvine, CA 92697, USA
  (\email{chenlong@math.uci.edu}).}
\and Xuehai Huang\thanks{Corresponding author. School of Mathematics, Shanghai University of Finance and Economics, Shanghai 200433, China 
  (\email{xuehaihuang@gmail.com}).}
}

\usepackage{amsopn}

\ifpdf
\hypersetup{
  pdftitle={Decoupling of Mixed Methods Based on Generalized Helmholtz Decompositions},
  pdfauthor={L. Chen and X. Huang}
}
\fi

\newcommand{\dd}{\,{\rm d}}
\newcommand{\bs}{\boldsymbol}

\DeclareMathOperator*{\img}{img}
\newcommand{\curl}{\operatorname{curl}}
\renewcommand{\div}{\operatorname{div}}
\newcommand{\grad}{\operatorname{grad}}
\DeclareMathOperator*{\tr}{tr}
\DeclareMathOperator*{\rot}{rot}

\newcommand{\dev}{\operatorname{dev}}
\newcommand{\sym}{\operatorname{sym}}
\newcommand{\skw}{\operatorname{skw}}
\newcommand{\spn}{\operatorname{spn}}

\newcommand{\XH}[1]{\textcolor{black}{#1}}
\newcommand{\LC}[1]{\textcolor{black}{#1}}

\begin{document}

\maketitle

\begin{abstract}
A framework to systematically decouple high order elliptic equations into combination of Poisson-type and Stokes-type equations is developed. 
The key is to systematically construct the underling commutative diagrams involving the complexes and Helmholtz decompositions in a general way.
Discretizing the decoupled formulation leads to a natural superconvergence between the Galerkin projection and the decoupled approximation. Examples include but not limit to: the primal formulations and mixed formulations of biharmonic equation, fourth order curl equation, and triharmonic equation etc. As a by-product, Helmholtz decompositions for many dual spaces are obtained.
\end{abstract}



\section{Introduction}

We shall develop a framework to systematically decouple high order elliptic equations into combination of Poisson-type and Stokes-type equations. 
The key is to systematically construct the underling commutative diagrams involving the complexes and Helmholtz decompositions in a general way.

Differential complexes and corresponding Helmholtz decompositions play the fundamental role in the design and analysis of mixed finite element methods.  Among many others, the de Rham complex for the Hodge Laplacian and the elasticity complex for the linear elasticity are two successful examples~\cite{ArnoldFalkWinther2006,ArnoldFalkWinther2010}.
A direct and useful result of a differential complex is the Helmholtz decomposition. With this decomposition, the kernel spaces of differential operators involved in the complex are characterized clearly.
The generalized Helmholtz decomposition of Banach spaces presented in this paper can be regarded as a generalization of the well-known Helmholtz-Hodge decomposition in \cite{ArnoldFalkWinther2006}. The Helmholtz-Hodge decomposition is $L^2$-orthogonal, while the generalized Helmholtz decomposition is only a direct sum, but not necessary to be orthogonal.

Our approach is based on the following diagram
\begin{equation}\label{intro:shortdiagrammixedform}
\begin{array}{c}
\xymatrix@R-=1.6pc{
& X \ar@{->}[r]^-{J_X} & X' & &\\
0\ar[r]^-{} & P \ar[r]^-{\dd ^-} & \Sigma \ar[r]^-{\dd}\ar@{}[u]|{\bigcup}
                & V' \ar[r]^-{} & 0\\
&&\widetilde{\Sigma}  \ar[u]^{\Pi_{\Sigma}} & V\ar[u]_{J_V}\ar[l]_{\Pi_V}}
\end{array},
\end{equation}
where the isomorphisms $J_X$ and $J_V$ are the inverses of the Riesz representations, and the rest linear operators are all continuous but not necessarily isomorphic.
The middle complex is exact, i.e. $\ker(\dd)=\img(\dd^{-})$, $\dd$ is surjective and $\dd^-$ is injective.
A stable Helmholtz decomposition can be derived from \eqref{intro:shortdiagrammixedform}
\begin{equation}\label{intro:abstracthelmholtzdecomp2}
\Sigma = \dd ^-P \oplus \Pi_{\Sigma} \Pi_V V.
\end{equation}
By finding the specific diagrams of \eqref{intro:shortdiagrammixedform},
we recover several well-known Helmholtz decompositions and discover many new Helmholtz decompositions.
In particular, we obtain Helmholtz decompositions for many dual spaces.
We summarize the details of these Helmholtz decompositions in Table~\ref{table:helmholtzdecomps}.
\begin{table}[htbp]
  \centering
  \caption{Examples of Helmholtz decomposition generated from \eqref{intro:shortdiagrammixedform}-\eqref{intro:abstracthelmholtzdecomp2}.}\label{table:helmholtzdecomps}
  \begin{tabular}{|c|c|c|}
    \hline
    Hilbert space & Helmholtz decomposition & Refs. \\ \hline
 $\boldsymbol L^2(\Omega)$ in 2d and 3d & $\nabla H_0^1(\Omega) \oplus^{\perp} \curl \bs H^1(\Omega)$ & \cite{ArnoldFalkWinther2006,GiraultRaviart1986, CarstensenBartelsJansche2002} \\ 
    $\boldsymbol{H}^{-1}(\Omega)$ in 2d & $\nabla L_0^2(\Omega) \oplus ^{\bot} \Delta (\curl H_0^2(\Omega))$ & \cite{OlshanskiiPetersReusken2006} \\ 
    $\boldsymbol H^{-1}(\div , \Omega)$ in $2d$ & $\nabla H_0^1(\Omega) \oplus \curl L^2(\Omega)$ & \cite{BrezziFortin1986} \\ 
    $\boldsymbol{H}^{-1}(\div \boldsymbol{\div },\Omega; \mathbb{S})$ in $2d$ & $\nabla ^s\times \boldsymbol{H}^{1}(\Omega; \mathbb{R}^2) \oplus \boldsymbol\pi H_0^1(\Omega)$ & \cite{KrendlRafetsederZulehner2016} \\ \hline \hline
    $\boldsymbol{H}(\div \boldsymbol{\div },\Omega; \mathbb{S})$ in $2d$ & $\nabla ^s\times \boldsymbol{H}^{1}(\Omega; \mathbb{R}^2) \oplus \boldsymbol\pi \Delta^{-1}L^2(\Omega)$ & Section \ref{subsec:div2HD} \\ 
    $\boldsymbol H^{-2}(\rot \mathbf{rot}, \Omega;\mathbb S)$ in $2d$ & $\boldsymbol \varepsilon \boldsymbol L^2(\Omega; \mathbb R^2) \oplus \bs\curl\curl  H_0^2(\Omega)$ & Section \ref{subsec:rot2HD} \\ 
    $\boldsymbol{H}^{-2}(\div^3, \Omega)$ in $2d$ & $\sym\bs\curl\boldsymbol{H}^1(\Omega; \mathbb S)\oplus\bs\Xi\nabla H_0^{2}(\Omega)$ & Section \ref{subsec:div3HD} \\ 
    $\boldsymbol H^{-1}(\div , \Omega)$ in $3d$ & $\nabla H_0^1(\Omega) \oplus \curl \bs L^2(\Omega; \mathbb{R}^3)$ & Section \ref{subsec:divm1HD} \\ 
    $\boldsymbol H^{-1}({\curl}, \Omega)$ in $3d$ & $\nabla L_0^2(\Omega) \oplus \curl \bs H_0(\curl, \Omega)$ & Section \ref{subsec:curlm1HD} \\ 
    $\boldsymbol{H}(\curl\curl, (K_0^c)')$ in $3d$ & $\spn^{-1}\skw\boldsymbol{H}^0(\bs\div, \dev\sym)\oplus K_0^c$ & Section \ref{subsec:curl2KdualHD} \\ \hline
  \end{tabular}
\end{table}

An abstract two-term mixed formulation based on the commutative diagram~\eqref{intro:shortdiagrammixedform} is: given $g\in\Sigma'$ and $f\in V'$,  find $(\sigma , u)\in \Sigma\times V$ such that
\begin{align}
(\sigma, \tau)_{X'}+ \langle \dd \tau, u\rangle & =\langle g, \tau\rangle \quad \forall~\tau\in\Sigma, \label{intro:mixed1}\\
\langle \dd \sigma, v\rangle & =\langle f, v\rangle  \quad \forall~v\in V.\label{intro:mixed2}
\end{align}
Under the assumption that the norm equivalence
\begin{equation}\label{Intro:normequivalenceSigma}
\|\tau\|_{\Sigma}^2\eqsim \|\tau\|_{X'}^2 + \|\dd \tau\|_{V'}^2 \quad \forall~\tau\in\Sigma
\end{equation}
holds, the mixed formulation \eqref{intro:mixed1}-\eqref{intro:mixed2} is wellposed.
Indeed, $\dd$ is surjective from the commutative diagram~\eqref{intro:shortdiagrammixedform}.
And the norm equivalence \eqref{Intro:normequivalenceSigma} can guarantee the continuity of the bilinear forms and the coercivity of $(\cdot, \cdot)_{X'}$ on $\ker(\dd)$.

To discretize the inner product $(\sigma, \tau)_{X'}$, we introduce $\phi=J_X^{-1}\sigma\in X$ and obtain an equivalent but unfolded three-term formulation:
find $(\phi, u, \sigma)\in X\times V\times\Sigma$ such that
\begin{align}
(\phi, \psi)_X - \langle \sigma, \dd 'v + \psi \rangle & =-\langle f, v\rangle  \quad\, \forall~(\psi, v)\in X\times V, \label{intro:mixedunfolded1} \\
\langle \dd 'u + \phi , \tau \rangle & =  \langle g, \tau\rangle \quad\quad \forall~\tau\in\Sigma. \label{intro:mixedunfolded2}
\end{align}
Equation \eqref{intro:mixedunfolded1} is the combination of \eqref{intro:mixed2} and $\phi=J_X^{-1}\sigma$, and equation \eqref{intro:mixedunfolded2} follows from \eqref{intro:mixed1} and $\phi=J_X^{-1}\sigma$.

Applying the Helmholtz decomposition \eqref{intro:abstracthelmholtzdecomp2} to the unfolded formulation~\eqref{intro:mixedunfolded1}-\eqref{intro:mixedunfolded2}, we obtain a decoupled formulation: find $w, u\in V$, $\phi\in X$, and $p\in P/\ker \dd ^-$ such that
\begin{align}
(w, v)_V&=\langle f, v\rangle  \quad\quad\quad\quad\quad\;\;\; \forall~v\in V, \label{intro:mixedunfoldedequiv1} \\
(\phi, \psi)_{X} - \langle \dd ^-p, \psi\rangle &= \langle\Pi_{\Sigma} \Pi_V w, \psi\rangle \quad\quad\;\; \forall~\psi\in X, \label{intro:mixedunfoldedequiv2} \\
\langle \dd ^-q, \phi\rangle &=\langle g, \dd ^-q\rangle \quad\quad\quad\quad\;\; \forall~q\in P/\ker \dd ^-, \label{intro:mixedunfoldedequiv3} \\
(u, \chi)_V&=\langle g-\phi, \Pi_{\Sigma} \Pi_V \chi\rangle \quad\, \forall~\chi\in V. \label{intro:mixedunfoldedequiv4}
\end{align}
The middle system \eqref{intro:mixedunfoldedequiv2}-\eqref{intro:mixedunfoldedequiv3} of $(\phi, p)$ is now a Stokes-type system,
and \eqref{intro:mixedunfoldedequiv1} and \eqref{intro:mixedunfoldedequiv4} are usually Poisson-type equations depending on the inner product $(\cdot, \cdot)_V$ .

By finding the underlying complexes, we recover some recent results on the decoupling of
\begin{itemize}
\item the HHJ method for plate problem into two Poisson equations and one linear elasticity problem \cite{KrendlRafetsederZulehner2016};

\item  the primal formulation of biharmonic equation in two and three dimensions into two Poisson equations and one Stokes-type equation \cite{HuangHuangXu2012,Huang2010, Gallistl2017};

\item the primal formulation of the fourth order elliptic singular perturbation problem into two Poisson equations and one Brinkman problem \cite{Gallistl2017};

\item  the primal formulation of fourth order $\curl$ equation into two Maxwell equations and one Stokes equation \cite{Zhang2016a}.
\end{itemize}
Moreover, we can get new decouplings by using the framework developed in this paper, for example, we decouple
\begin{itemize}
\item the mixed formulation of fourth order $\curl$ equation into two Maxwell equations and one mixed formulation of Poisson-type equation in Section~\ref{sec:quadcurlmixed};

\item the primal formulation of the triharmonic equation in two dimensions into two biharmonic equations and one tensorial Stokes equation in Section~\ref{sec:triharmonicprimal};

\item the mixed formulation of the triharmonic equation in two dimensions into two biharmonic equations and one tensorial Poisson-type equation in Section~\ref{sec:triharmonicmixed};

\item the $m$-th harmonic equation into two $(m-1)$-th harmonic equations and one tensorial Stokes-type equation in Section~\ref{sec:triharmonicprimal}.
\end{itemize}

Compared to the original formulation, it is much easier to construct conforming finite element spaces for the decoupled formulation,
since the order of the system is reduced and the finite element methods for Stokes equation and Poisson equation are well-developed.
We shall also show a natural superconvergence between the Galerkin projection and the approximation based on the decoupled formulation.

We are motivated by the pioneer work of such decoupling for Reissner-Mindlin plate model~\cite{BrezziFortin1986}, which has been found since 1980s', and a recent decomposition for HHJ formulation of Kirchhoff plate model~\cite{KrendlRafetsederZulehner2016,RafetsederZulehner2017} and biharmonic equation in three dimensions \cite{PaulyZulehner2016}. \XH{Results can be also found for the primal formulation of biharmonic equations~\cite{HuangHuangXu2012,Huang2010, Gallistl2017}, the primal formulation of the fourth order $\curl$ equation~\cite{Zhang2016a,BrennerSunSung2017}, the eigenvalue problem of biharmonic equation~\cite{ZhangXiJi2018} and the linear second-order elliptic problem in nondivergence form~\cite{Gallistl2017a}.}
Viewed as the polyharmonic generalized Stokes problem, the $m$-th harmonic equation was decoupled into $(2m-2)$ Poisson-type problems and one generalized Stokes equation over the symmetric tensors by applying a split recursively in~\cite{Gallistl2017}. As comparison, we decouple the $m$-th harmonic equation into two $(m-1)$-th harmonic equations and one tensorial Stokes-type equation independently, \XH{and recursively to decouple into $2^{m-1}$ Poisson equations and $(2^{m-1}-1)$ Stokes-type equations}, which is different from the decoupling in \cite{Gallistl2017} for $m\geq3$.
\LC{We can also stop the decoupling at biharmonic equations, which can be discretized directly by many existing finite elements methods.}
We refer to~\cite{Schedensack2016,Zhang2018} for more works on reducing the $m$-th harmonic equation into the lower order partial differential equations.
Our framework unifies most of those results and will lead to many more \XH{Hodge decompositions} especially for high order elliptic equations. Along this way, we can decouple the higher order partial differential equation into lower order ones, which makes the discretization easier.

In addition, the explicit expression of the kernel space can be used to develop fast solvers, see, for example,~\cite{Hiptmair1997,ArnoldFalkWinther2000,HiptmairXu2007,Chen2016,ChenHuHuang2015a}.
The Helmholtz decomposition is also a key tool to construct the {\em a posteriori} error estimators of nonconforming and mixed finite element methods~\cite{Alonso1996,Carstensen2005,CarstensenBartels2002,CarstensenHu2007,HuangHuangXu2011,ChenHuHuangMan2017}.
The important role of the structure revealed in our work for designing fast solvers and the {\em a posterior} error analysis will be explored somewhere else.


The rest of this paper is organized as follows. In Section 2, we establish the generalized Helmholtz decomposition
based on the commutative diagram and give several examples.
The abstract mixed formulation and its decomposition based on the Helmholtz decomposition are presented in Section 3.
In Section 4, we discretize the decoupled formulation directly illustrated by two examples.
Throughout this paper, we use ``$\lesssim\cdots $" to mean that ``$\leq C\cdots$", where $C$ is a generic positive constant independent of meshsize $h$, which may take different values at different appearances.
And $a\eqsim b$ means $a\lesssim b$ and $b\lesssim a$.

\section{Generalized Helmholtz Decompositions}
In this section we apply the splitting lemma of Banach spaces and Hilbert spaces to differential complexes and obtain Helmholtz-type decompositions for several Sobolev spaces with negative index.
\subsection{Background in Functional Analysis}
We start from a short exact sequence
\begin{equation}
\label{eq:shorttilde}
\widetilde{W}\; \autorightarrow{$\tilde{\dd}_2$}{}\; \widetilde{V} \; \autorightarrow{$\tilde{\dd}_1$}{} \;\widetilde{U} \autorightarrow{}{} 0.
\end{equation}
Here capital letters represent Banach spaces and $\tilde{\dd}_i \, (i=1,2)$ are bounded linear operators. The sequence \eqref{eq:shorttilde} is exact meaning that
$$
\ker(\tilde{\dd}_1) = \img(\tilde{\dd}_2), \; \img(\tilde{\dd}_1) = \widetilde{U}.
$$
The space $\widetilde{W}$ can be further reduced to the quotient space $\widetilde{W}/\ker(\tilde{\dd}_2)$ so that
\begin{equation*}
0\; \autorightarrow{}{} \widetilde{W}/\ker(\tilde{\dd}_2) \; \autorightarrow{$\tilde{\dd}_2$}{}\; \widetilde{V} \; \autorightarrow{$\tilde{\dd}_1$}{} \;\widetilde{U} \autorightarrow{}{} \; 0
\end{equation*}
forms a short exact sequence in the context of group \cite{Hatcher2002}. 

Let $U,V$ be two additional Banach spaces and $\dd _1: U\to V$ be a bounded linear operator.
Let $I_V: \dd_1 U \to\widetilde{V}$ be a bounded linear operator, and $J_U: U\to\widetilde{U}$ be an isomorphism satisfying the assumption:
\begin{equation}\label{eq:shortdiagramequality}
\tilde{\dd}_1I_V \dd _1 u = J_Uu\quad  \text{for all } u\in U,
\end{equation}
which can be summarized as the following commutative diagram
\begin{equation}\label{eq:shortdiagram}
\begin{array}{c}
\xymatrix{
\widetilde{W} \ar[r]^-{\tilde{\dd}_2} & \widetilde{V} \ar[r]^-{\tilde{\dd}_1}
                & \widetilde{U} \ar[r]^-{} & 0 \\
&V  \ar[u]^{I_V} & U\ar[u]_{J_U}\ar[l]_{\dd _1}}
\end{array}.
\end{equation}

To derive the generalized Helmholtz decomposition, we first recall the splitting lemma in algebraic topology.
\begin{lemma}[Splitting lemma in \cite{Hatcher2002}]\label{lem:splittinglemma}
For a short exact sequence $$0\; \autorightarrow{}{} U \; \autorightarrow{$\dd_1$}{}\; V \; \autorightarrow{$\dd_2$}{} \;W \autorightarrow{}{} \; 0$$ of abelian groups the following statements are equivalent:
\begin{enumerate}[(a)]
\item There is a homomorphism $\dd_3: V\to U$ such that $\dd_3\dd_1$ is the identity on $U$.
\item There is a homomorphism $\dd_4: W\to V$ such that $\dd_2\dd_4$ is the identity on $W$.
\item The group $V$ is isomorphic to the direct sum of $U$ and $W$, with $\dd_1$ corresponding to the natural injection of $U$ and $\dd_2$ to the natural projection onto $W$.
\end{enumerate}
\end{lemma}

Apparently Banach spaces are abelian groups under addition.
Then we have an generalized Helmholtz decomposition as follows.
\begin{theorem}\label{thm:abstracthelmholtzdecomp}
Suppose we have a short exact sequence \eqref{eq:shorttilde}. Assume the commutative diagram \eqref{eq:shortdiagram} holds with all the linear operators being bounded and $J_U: U\to\widetilde{U}$ being an isomorphism. Then we have a stable Helmholtz decomposition
\begin{equation}\label{eq:splitting}
\widetilde{V} = \tilde{\dd}_2 (\widetilde{W}/\ker(\tilde{\dd}_2)) \oplus I_V \dd _1 U,
\end{equation}
where $\oplus$ means the direct sum.
More precisely for any $\widetilde{v}\in \widetilde{V}$, there exist $\widetilde{w}\in \widetilde{W}/\ker \tilde{\dd}_2$ and $u\in U$ such that
\begin{align}\label{eq:abstracthelmholtzdecomp1}
\widetilde{v} = \tilde{\dd}_2 \widetilde{w} + I_V \dd _1 u,\\
\label{eq:abstracthelmholtzdecomp1stable}
\|\widetilde{w}\|_{\widetilde{W}}+\|u\|_U\lesssim \|\widetilde{v}\|_{\widetilde{V}}.
\end{align}
\end{theorem}
\begin{proof}
By assumptions, we have a short exact sequence
\begin{equation*}
0 \; \autorightarrow{}{} \widetilde{W}/\ker(\tilde{\dd}_2) \; \autorightarrow{$\tilde{\dd}_2$}{}\; \widetilde{V} \; \autorightarrow{$J_U^{-1}\tilde{\dd}_1$}{} \;U \autorightarrow{}{} \; 0.
\end{equation*}
And $I_V \dd _1$ forms a right inverse of $J_U^{-1}\tilde{\dd}_1$. Thus \eqref{eq:splitting} holds by Lemma~\ref{lem:splittinglemma} which is stable as all operators involved are continuous.
\end{proof}


The generalized Helmholtz decomposition \eqref{eq:splitting} is a direct result of the splitting lemma after finding the underlying commutative diagram \eqref{eq:shortdiagram}. One contribution of this paper is to construct various commutative diagrams to induce stable Helmholtz decompositions on Sobolev spaces of negative order.

\begin{remark}\label{rmk:helmdecomHarmonic}
Consider Hilbert space $\widetilde{V}$ with inner product $(\cdot, \cdot)_{\widetilde{V}}$.
The quotient space $\ker(\tilde{\dd}_1)/\tilde{\dd}_2 (\widetilde{W})$ is isomorphic to the space of harmonic forms
\[
\mathfrak H:=\left\{\widetilde{v}\in\widetilde{V}: \tilde{\dd}_1\widetilde{v}=0,\; (\widetilde{v}, \tilde{\dd}_2 \widetilde{w})_{\widetilde{V}}=0\quad \forall~\widetilde{w}\in\widetilde{W}\right\}.
\]
When the quotient space $\ker(\tilde{\dd}_1)/\tilde{\dd}_2 (\widetilde{W})$ is non-trivial, i.e., the sequence \eqref{eq:shorttilde} is not exact,
the Helmholtz decomposition will be
\begin{equation*}
\widetilde{V} = \tilde{\dd}_2 (\widetilde{W}/\ker(\tilde{\dd}_2)) \oplus I_V \dd _1 U \oplus \mathfrak H.
\end{equation*} $\hfill\Box$
\end{remark}

In Theorem \ref{thm:abstracthelmholtzdecomp}, the decomposition is a direct sum but not necessarily orthogonal. Indeed in the proof we do not use the inner product structure. We now explore the orthogonality for Hilbert complexes.
In what follows, we always denote by $\langle \cdot, \cdot \rangle$ the duality pairing and reserve $( \cdot, \cdot )$ for the $L^2$ inner product.

Denoted by $X'$ the dual space of a linear space $X$ and $T': Y'\to X'$ the dual of a linear operator $T: X\to Y$ defined as
 $$
 \langle T' g, x \rangle : = \langle g, Tx \rangle.
 $$
When $X$ is a Hilbert space with an inner product $(\cdot, \cdot)_X$ and $X'$ is the continuous dual of $X$, by Riesz representation theorem, we have an isomorphism $J_X: X\to X'$: for any $w\in X$, define $J_Xw\in X'$ as
\begin{equation}\label{eq:J_X}
\langle J_Xw, v\rangle=(w, v)_X \quad \forall~v\in X.
\end{equation}
The induced inner product and norm for any $w', v'\in X'$ are given by
\begin{align}
(w', v')_{X'}&:=(J_X^{-1}w', J_X^{-1}v')_{X} = \langle J_X^{-1}w', v' \rangle= \langle w', J_X^{-1}v' \rangle, \label{eq:Xdualinnerproduct}\\
\|w'\|_{X'}&:=\|J_X^{-1}w'\|_{X}. \notag
\end{align}

Let $U, V, W$ be Hilbert spaces. Suppose we have a short exact sequence 
of their dual spaces
\begin{equation}\label{eq:shortdual}
0\; \autorightarrow{}{} W'/\ker(\dd_2')\; \autorightarrow{$\dd _2'$}{}\; V' \; \autorightarrow{$\dd _1'$}{} \;U'  \autorightarrow{}{} \; 0.
\end{equation}
By Remark 2.15 in \cite{PaulyZulehner2016},  the dual complex \eqref{eq:shortdual} implies the exact sequence
\begin{equation}\label{eq:UVW}
0\; \autorightarrow{}{} U\; \autorightarrow{$\dd _1$}{}\; V \; \autorightarrow{$\dd _2$}{} \;W.
\end{equation}
Then it is apparent that $\dd_1'J_V \dd_1$ is an isomorphism from $U$ to $U'$.
By taking $J_U=\dd_1'J_V \dd_1$, the assumption \eqref{eq:shortdiagramequality} holds.
With $\widetilde{X} \, (X=U, V, W)$ and $\tilde{\dd}_i \, (i=1,2)$ replaced by $X' \, (X=U, V, W)$ and $\dd_i'  \, (i=1,2)$,  the commutative diagram \eqref{eq:shortdiagram} becomes
\begin{equation}\label{eq:shortdiagramdual}
\begin{array}{c}
\xymatrix{
W' \ar[r]^-{\dd_2'} & V' \ar[r]^-{\dd_1'}
                & U' \ar[r]^-{} & 0 \\
&V  \ar[u]^{J_V} & U\ar[u]_{J_U}\ar[l]_{\dd_1}}
\end{array}.
\end{equation}
Applying Theorem~\ref{thm:abstracthelmholtzdecomp} to the commutative diagram \eqref{eq:shortdiagramdual}, we recover the following orthogonal Helmholtz decomposition.
\begin{corollary}\label{cor:orthdec}
Suppose the short exact Hilbert sequence \eqref{eq:shortdual} holds.
Then we have the $(\cdot,\cdot)_{V'}$-orthogonal Helmholtz decomposition
\begin{equation}\label{eq:HV}
 V' = \dd _2' (W'/\ker(\dd_2')) \oplus^{\bot} J_V \dd _1 U.
\end{equation}
That is for any $v'\in V'$, there exist $w'\in W'/\ker \dd _2'$ and $u\in U$ such that
\begin{align}\label{eq:abstracthelmholtzdecomp1dual}
v' &= \dd _2' w' + J_V \dd _1 u,\\
\label{eq:abstracthelmholtzdecomp1stabledual}
 \|v'\|_{V'}^2 & = \|\dd _2'w'\|_{V'}^2 +\|J_V \dd _1 u\|_{V'}^2.
\end{align}
\end{corollary}

Corollary~\ref{cor:orthdec} is indeed one way to express the well-known closed range theorem \cite{Yosida1980} for Hilbert spaces and the associated orthogonal space decomposition.

In the decomposition \eqref{eq:HV}, we need to know $J_V$ which involves the inner product of Hilbert space $V$. Sometimes we do not know exactly the space $V$ or not necessarily need to know. We give an example to illustrate this point.

Consider a dense subspace $V$ of a larger space $Y$ endowed with the inner product $(\cdot,\cdot)_Y$. In most places in this paper, $Y$ is the $L^2$ space for scalar or vector functions with $(\cdot,\cdot)_Y = (\cdot,\cdot)$ being the $L^2$-inner product.
We can equip $V$ with the graph inner product
\begin{equation}\label{eq:graphnorm}
(w, v)_{V}:=(w, v)_Y + (\dd _2w, \dd _2v)_W.
\end{equation}
Or we can start from $Y$ and define $V$ as the subspace of $Y$ with $\|\cdot\|_{V}< \infty$.
By identifying $Y'$ with $Y$ using the inner product $(\cdot,\cdot)_Y$, we have the rigged Hilbert space \cite{Brezis2011, MbyadridModino2001} 
\begin{equation}\label{eq:riggedHilbert}
V\subset Y\subset V'.
\end{equation}
We compute $J_{V}: V \to V'$ as: for any $u\in U$ and $v\in V$
\[
\langle J_{V}\dd _1u, v\rangle=(\dd _1u, v)_{V}=(\dd _1u, v)_Y+(\dd _2\dd _1u, \dd _2v)_W=(\dd _1u, v)_Y.
\]
Thus $J_V$ is the composition of the natural inclusions in \eqref{eq:riggedHilbert} on $\dd _1U$. On the other hand, we can use the commutative diagram to characterize the dual space.

\begin{corollary}\label{cor:dualcharac}
Suppose the short exact Hilbert sequence \eqref{eq:shortdual} holds, and $V$ is equipped with the graph inner product \eqref{eq:graphnorm}.
Assume we have another short exact Hilbert sequence
\[
W'\; \autorightarrow{$\dd _2'$}{}\; \widetilde V \; \autorightarrow{$\dd _1'$}{} \;U'  \autorightarrow{}{} \; 0,
\]
and commutative diagram
\begin{equation*}
\begin{array}{c}
\xymatrix{
W' \ar[r]^-{\dd_2'} & \widetilde{V} \ar[r]^-{\dd_1'}
                & U' \ar[r]^-{} & 0 \\
&V  \ar[u]^{I} & U\ar[u]_{J_U}\ar[l]_{\dd_1}}
\end{array}
\end{equation*}
with $I$ being the embedding operator. Then $\widetilde{V}=V'=\dd _2' (W'/\ker(\dd_2')) \oplus \dd _1 U$.
\end{corollary}
\begin{proof}
The result immediately follows from the previous illustration, Corollary~\ref{cor:orthdec} and Theorem~\ref{thm:abstracthelmholtzdecomp}.
\end{proof}





We shall present examples in the sequel. Let $\Omega \subset \mathbb R^n, n = 2, 3,$ be a bounded Lipschitz domain.
Denote by $\mathbb{M}$ the space of all $n\times n$ tensors, $\mathbb{S}$ the space of all symmetric $n\times n$ tensors, and $\mathbb{K}$ the space of all skew-symmetric $n\times n$ tensors. For any tensor $\bs\tau\in\mathbb{M}$, let $\sym\bs\tau:=(\bs\tau+\bs\tau^{\intercal})/2$ be the symmetric part of the tensor, and $\skw\bs\tau:=(\bs\tau-\bs\tau^{\intercal})/2$ be the skew-symmetric part. Denote the deviatoric part and the trace of the tensor $\bs\tau$ by $\dev\bs\tau$ and $\tr\bs\tau$ accordingly.
We have
\[
\dev\bs\tau=\bs\tau-\frac{1}{n}(\tr\bs\tau)\bs I.
\]
Define operator $\spn: \mathbb R^3\to \mathbb M$ as follows: for any vector $\bs a=(a_1, a_2, a_3)^{\intercal}\in \mathbb R^3$, the tensor $\spn\bs a\in\mathbb M$ is given by
\[
\spn \bs a :=
\begin{pmatrix}
0 & -a_3 & a_2 \\
a_3 & 0 & -a_1 \\
-a_2 & a_1 & 0
\end{pmatrix}.
\]
Denote by $\mathfrak{S}_3$ the set of all permutations of $(1, 2, 3)$.
Define the set of symmetric third-order tensors as (cf. \cite[Section 2]{Schedensack2016})
\[
\mathbb S(3):=\{\bs\tau\in(\mathbb R^2)^3 : \bs\tau_{j_1,j_2,j_3}=\bs\tau_{j_{\sigma(1)}, j_{\sigma(2)}, j_{\sigma(3)}}\quad\forall~(j_1, j_2, j_3)\in\{1, 2\}^3,\, \sigma\in\mathfrak{S}_3\}.
\]
The symmetric part $\sym \bs\tau\in \mathbb S(3)$ of a tensor $\bs\tau\in (\mathbb R^2)^3$ is defined by
\[
(\sym\bs\tau)_{j_1,j_2,j_3}:=\frac{1}{\#(\mathbb S(3))}\sum_{\sigma\in\mathbb S(3)}\bs\tau_{j_{\sigma(1)}, j_{\sigma(2)}, j_{\sigma(3)}}
\]
for all $(j_1, j_2, j_3)\in\{1, 2\}^3$.
We use standard notation for Sobolev spaces and boldface letters for vector and tensor valued spaces. When we want to emphasize the spatial dimension, we include $\mathbb R^n$ into the notations of spaces.

Recall the de Rham complexes in two dimensions
\begin{align}
\quad\; 0\;\autorightarrow{}{} \;H_{0}^1(\Omega)\;\autorightarrow{$\curl$}{} \; \boldsymbol H_{0}(\div , \Omega)\; \autorightarrow{$\div $}{} \;L_0^2(\Omega)\;\autorightarrow{}{}\;0,\label{eq:deRhamcomplex2d-div0} \quad\;\\
\quad\;\; \mathbb R\;\autorightarrow{}{} \;H^1(\Omega)\;\autorightarrow{$\curl$}{} \; \boldsymbol H(\div , \Omega)\; \autorightarrow{$\div $}{} \;L^2(\Omega)\;\autorightarrow{}{}\;0,\label{eq:deRhamcomplex2d-div} \quad\;\; \\
0\;\autorightarrow{}{} \;H_{0}^{s+2}(\Omega)\;\autorightarrow{$\curl$}{} \; \boldsymbol H_{0}^{s+1}(\Omega; \mathbb{R}^2)\; \autorightarrow{$\div$}{} \;H_{0}^{s}(\Omega)\;\autorightarrow{}{}\; 0,\label{eq:deRhamcomplex2d-H0} \\
\mathbb R\;\autorightarrow{}{} \;H^{s+2}(\Omega)\;\autorightarrow{$\curl$}{} \; \boldsymbol H^{s+1}(\Omega; \mathbb{R}^2)\; \autorightarrow{$\div$}{} \;H^{s}(\Omega)\;\autorightarrow{}{}\;0, \label{eq:deRhamcomplex2d-H}
\end{align}
and the de Rham complexes in three dimensions
\begin{align}
0\autorightarrow{}{} H_{0}^1(\Omega)\autorightarrow{$\grad$}{} \bs H_0(\curl, \Omega) \autorightarrow{$\curl$}{} \bs H_0(\div , \Omega) \autorightarrow{$\div $}{} L_0^2(\Omega)\autorightarrow{}{} 0, \label{eq:deRhamcomplex3d-0} \\
\mathbb R\autorightarrow{}{} H^1(\Omega)\autorightarrow{$\grad$}{} \bs H(\curl, \Omega) \autorightarrow{$\curl$}{} \bs H(\div , \Omega) \autorightarrow{$\div $}{} L^2(\Omega)\autorightarrow{}{}0, \label{eq:deRhamcomplex3d} \\
\resizebox{.9\hsize}{!}{$
0\autorightarrow{}{} H_0^{s+3}(\Omega)\autorightarrow{$\grad$}{} \bs H_0^{s+2}(\Omega; \mathbb{R}^3) \autorightarrow{$\curl$}{} \bs H_0^{s+1}(\Omega; \mathbb{R}^3) \autorightarrow{$\div $}{} H_0^{s}(\Omega)\autorightarrow{}{} 0,$} \label{eq:deRhamcomplex3d-H0}\\
\resizebox{.9\hsize}{!}{$
\mathbb R\autorightarrow{}{} H^{s+3}(\Omega)\autorightarrow{$\grad$}{} \bs H^{s+2}(\Omega; \mathbb{R}^3) \autorightarrow{$\curl$}{} \bs H^{s+1}(\Omega; \mathbb{R}^3) \autorightarrow{$\div $}{} H^{s}(\Omega)\autorightarrow{}{}0,$} \label{eq:deRhamcomplex3d-H}
\end{align}
with $s\in\mathbb R$. In 2-D, as $\curl$ is a rotation of $\grad$ operator, we could have similar sequences by replacing $\curl$ by $\grad$, $\div$ by $\rot$, and $\bs H(\div, \Omega)$ space by $\bs H(\rot, \Omega)$ space. For example, an analogue of \eqref{eq:deRhamcomplex2d-H} is
\begin{equation}
\mathbb R\;\autorightarrow{}{} \;H^{s+2}(\Omega)\;\autorightarrow{$\grad$}{} \; \boldsymbol H^{s+1}(\Omega; \mathbb{R}^2)\; \autorightarrow{$\rot$}{} \;H^{s}(\Omega)\;\autorightarrow{}{}\;0. \label{eq:deRhamcomplex2d-H-rot}
\end{equation}

When $\Omega$ is simply connected with connected boundary, the $L^2$ de Rham complexes \eqref{eq:deRhamcomplex2d-div0}-\eqref{eq:deRhamcomplex2d-div} and \eqref{eq:deRhamcomplex3d-0}-\eqref{eq:deRhamcomplex3d} are exact \cite{GiraultRaviart1986, ArnoldFalkWinther2006, ArnoldFalkWinther2010}, the complexes \eqref{eq:deRhamcomplex2d-H} and \eqref{eq:deRhamcomplex3d-H}-\eqref{eq:deRhamcomplex2d-H-rot}
are exact if $s$ is an integer, and the complexes \eqref{eq:deRhamcomplex2d-H0} and \eqref{eq:deRhamcomplex3d-H0}
are exact if $s$ is a nonnegative integer \cite{PaulyZulehner2016, CostabelMcIntosh2010}.
When $\Omega$ is a bounded domain starlike with respect to a ball, the complexes \eqref{eq:deRhamcomplex2d-H} and \eqref{eq:deRhamcomplex3d-H}-\eqref{eq:deRhamcomplex2d-H-rot}
are exact for any $s\in\mathbb R$, and the complexes \eqref{eq:deRhamcomplex2d-H0} and \eqref{eq:deRhamcomplex3d-H0}
are exact if $s$ is nonnegative and $s-\frac{1}{2}$ is not an integer \cite[p. 301]{CostabelMcIntosh2010}.
Hereafter, we only consider the domain $\Omega$ without harmonic forms, i.e.,
we always assume the bounded Lipschitz domain $\Omega$ is simply connected with connected boundary in this paper.
It will be more sophisticated in the discretization of decoupled formulations and in designing fast solvers and adaptive algorithms when the non-trivial harmonic forms appear in the Helmholtz decompositions.
For ease of presentation, we use $\bs H^s(\Omega)$ to denote $\bs H^s(\Omega; \mathbb{R}^n)$ for $n=1, 2, 3$.

We recall the well known
$L^2$-orthogonal Helmholtz decomposition (cf.~\cite{Helmholtz1858, ArnoldFalkWinther2006})
\[
\boldsymbol L^2(\Omega) = \nabla H_0^1(\Omega) \oplus^{\perp} \curl\left(\bs H^1(\Omega)/\ker(\curl)\right)
\]
in two and three dimensions from the exact sequences \eqref{eq:deRhamcomplex2d-H} and \eqref{eq:deRhamcomplex3d-H} with $s=-1$, and
the $H^{-1}$-orthogonal decomposition of $\boldsymbol H^{-1}(\Omega; \mathbb{R}^2)$ (cf.~\cite[Lemma 2.4]{OlshanskiiPetersReusken2006})
\[
\boldsymbol{H}^{-1}(\Omega; \mathbb{R}^2) = \nabla L_0^2(\Omega) \oplus ^{\bot} \Delta (\curl H_0^2(\Omega))
\]
from the exact sequences \eqref{eq:deRhamcomplex2d-H-rot} with $s=-2$.
In the following, we present several less well-known and some new Helmholtz decompositions.

\subsection{Helmholtz decomposition of $\bs H^{-1}({\rm div})$ space}\label{subsec:divm1HD}
For $n=2$ and $3$,
define
\[
\boldsymbol H^{-1}(\div , \Omega):=\{\boldsymbol{\phi}\in \boldsymbol{H}^{-1}(\Omega): \div \boldsymbol{\phi}\in H^{-1}(\Omega)\}
\]
with squared norm
$
\|\boldsymbol{\phi}\|_{H^{-1}(\div )}^2:=\|\boldsymbol{\phi}\|_{-1}^2+\|\div \boldsymbol{\phi}\|_{-1}^2.
$
\begin{lemma}
The complex
\[
\bs L^2(\Omega)\;\autorightarrow{$\curl$}{} \; \boldsymbol H^{-1}({\rm div}, \Omega)\; \autorightarrow{$\div$}{} \;H^{-1}(\Omega)\;\autorightarrow{}{}\;0
\]
is exact in both two and three dimensions. 
\end{lemma}
\begin{proof}
We know $\ker(\div)\cap \bs H^{-1}(\Omega) = \curl (\bs L^2 (\Omega))$ from the exact sequence \eqref{eq:deRhamcomplex2d-H} or \eqref{eq:deRhamcomplex3d-H} with $s = -2$. Obviously $\ker(\div)\cap \bs H^{-1}(\div, \Omega) = \ker(\div)\cap \bs H^{-1}(\Omega)$.

With $s = -1$, it holds $\div\boldsymbol L^2(\Omega; \mathbb{R}^n)=H^{-1}(\Omega)$, which together with $\boldsymbol L^2(\Omega; \mathbb{R}^n)\subset\boldsymbol H^{-1}({\rm div}, \Omega)$ indicates $\div\boldsymbol H^{-1}({\rm div}, \Omega)=H^{-1}(\Omega)$.
\end{proof}

With this exact sequence, we build up the commutative diagram
\begin{equation}\label{Hdiv_1diagram}
\begin{array}{c}
\xymatrix{
\bs L^2(\Omega) \ar[r]^-{\curl} & \boldsymbol H^{-1}({\rm div}, \Omega) \ar[r]^-{\rm div}
                & H^{-1} (\Omega) \ar[r]^-{} & 0\\
&\boldsymbol H_0(\curl , \Omega) \ar[u]^{I} & H_0^1(\Omega)\ar[u]_{\Delta}\ar[l]_-{\grad}}
\end{array}.
\end{equation}
By Theorem~\ref{thm:abstracthelmholtzdecomp},
we obtain the Helmholtz decomposition in both two and three dimensions
\begin{equation}\label{eq:Hdiv_1helmholtzdecomp}
\boldsymbol H^{-1}(\div , \Omega) = \nabla H_0^1(\Omega) \oplus \curl\left(\bs L^2(\Omega)/\ker(\curl)\right).
\end{equation}
In two dimensions, $\ker(\curl) = \mathbb R$ and thus the Helmholtz decomposition \eqref{eq:Hdiv_1helmholtzdecomp} reads as
\[
\boldsymbol H^{-1}(\div , \Omega) = \nabla H_0^1(\Omega) \oplus \curl L_0^2(\Omega)
\]
which has been presented in \cite[Proposition 2.3]{BrezziFortin1986}. In three dimensions, $\ker(\curl) = \nabla H^1(\Omega)$, it becomes
\[
\boldsymbol H^{-1}(\div , \Omega) = \nabla H_0^1(\Omega) \oplus \curl\left(\bs L^2(\Omega; \mathbb{R}^3)/\nabla H^1(\Omega)\right).
\]

Applying Remark 2.15 in \cite{PaulyZulehner2016} to the short exact sequence
\[
0\; \autorightarrow{}{} H_0^{1}(\Omega)\; \autorightarrow{$\grad$}{}\; \boldsymbol H_0(\curl , \Omega) \; \autorightarrow{$\curl$}{} \;\bs L^2(\Omega),
\]
we know the dual complex
\[
\bs L^2(\Omega)\;\autorightarrow{$\curl$}{} \; (\boldsymbol H_0(\curl , \Omega))^{\prime}\; \autorightarrow{$\div$}{} \;H^{-1}(\Omega)\;\autorightarrow{}{}\;0
\]
is also exact.
Then by Corollary~\ref{cor:dualcharac}, we get from the commutative diagram \eqref{Hdiv_1diagram} that
\begin{equation}\label{eq:curldiv_1spaceequiv}
(\boldsymbol H_0(\curl , \Omega))^{\prime}=\boldsymbol H^{-1}({\rm div}, \Omega).
\end{equation}
The inclusion $\boldsymbol H^{-1}({\rm div}, \Omega)\subset(\boldsymbol H_0(\curl , \Omega))^{\prime}$ has been proved in the book~\cite[p. 338]{Braess2007} for $n=2$.

%

As we mentioned early that $J_{H(\curl)}$ is just identity operator on $\nabla H_0^1(\Omega)$, thus the decomposition \eqref{eq:Hdiv_1helmholtzdecomp} is orthogonal in the inner product $(\cdot,\cdot)_{H(\curl)'}$, but not in the $L^2$ inner product nor in the $H^{-1}$ inner product. 

%
%

%


\subsection{Helmholtz decomposition of $\bs H^{-1}({\curl})$ space}\label{subsec:curlm1HD}
%
Following \cite{ChenWuZhongZhou2016}, we introduce the space
\[
K_0^c:=\{\boldsymbol{\phi}\in \bs H_0(\curl, \Omega): \mathrm{div}\boldsymbol{\phi}=0\}=\bs H_0(\curl, \Omega)/\grad H_0^1
\]
equipped with norm $\|\cdot\|_{H(\curl)}$.
Noting that $\curl K_0^c=\curl\bs H_0(\curl, \Omega)$, we get the following exact sequence from the 3D de Rham complex \eqref{eq:deRhamcomplex3d-0}
\[
0\; \autorightarrow{}{}K_0^c \autorightarrow{$\curl$}{} \bs H_0(\div , \Omega) \autorightarrow{$\div $}{} L_0^2(\Omega)\autorightarrow{}{} \; 0,
\]
which together with Remark 2.15 in \cite{PaulyZulehner2016} implies the exactness of the dual complex
\[
0\autorightarrow{}{}L_0^2(\Omega)\;\autorightarrow{$\grad$}{} \; \boldsymbol (\bs H_0(\div , \Omega))' \; \autorightarrow{$\curl$}{} \;(K_0^c)'\;\autorightarrow{}{}\;0.
\]

We then construct an exact sequence with dual spaces to characterize the dual space $(\bs H_0(\div , \Omega))'$.
Define
$$
\boldsymbol H^{-1}(\curl, \Omega) = \{\boldsymbol{\phi}\in \boldsymbol{H}^{-1}(\Omega; \mathbb{R}^3): \curl\boldsymbol{\phi}\in \boldsymbol{H}^{-1}(\Omega; \mathbb{R}^3)\}
$$
with squared norm
$
\|\boldsymbol{\phi}\|_{H^{-1}(\curl )}^2:=\|\boldsymbol{\phi}\|_{-1}^2+\|\curl \boldsymbol{\phi}\|_{-1}^2.
$
\begin{lemma}\label{lem:temp20180205}
The complex
\[
0\autorightarrow{}{}L_0^2(\Omega)\;\autorightarrow{$\grad$}{} \; \boldsymbol H^{-1}({\curl}, \Omega)\; \autorightarrow{$\curl$}{} \;(K_0^c)'\;\autorightarrow{}{}\;0
\]
is exact.
\end{lemma}
\begin{proof}
We have $\ker(\curl)\cap\boldsymbol{H}^{-1}(\Omega; \mathbb{R}^3)=\grad L_0^2(\Omega)$ by taking $s=-3$ in the exact sequence \eqref{eq:deRhamcomplex3d-H}.
Apparently $\ker(\curl)\cap\boldsymbol H^{-1}({\curl}, \Omega)=\ker(\curl)\cap\boldsymbol{H}^{-1}(\Omega; \mathbb{R}^3)$, which implies the exactness of the former complex.

By the Poincar\'e inequality on $K_0^c$ \cite{Monk2003,Hiptmair2002}, $(\curl \cdot, \curl \cdot)$ defines an inner product on $K_0^c$ and $(\curl \curl)^{-1}: (K_0^c)' \to K_0^c$ is an isomorphism. Given a $\bs f\in (K_0^c)'$, finding $\bs u\in K_0^c$ such that $\curl\curl \bs u = \bs f$ in $(K_0^c)'$ is the Maxwell's equation with divergence free constraint. Since $\curl K_0^c\subset \boldsymbol H^{-1}({\curl}, \Omega)$, we obtain
$$
(K_0^c)'=\curl\curl K_0^c\subset\curl\boldsymbol H^{-1}(\curl, \Omega).
$$
Due to \eqref{eq:curldiv_1spaceequiv},
\[
\curl\boldsymbol H^{-1}(\curl, \Omega)\subset \boldsymbol H^{-1}({\rm div}, \Omega)=(\boldsymbol H_0(\curl , \Omega))^{\prime}\subset (K_0^c)'.
\]
Therefore
\begin{equation}\label{eq:temp20180205}
(K_0^c)'=\curl\boldsymbol H^{-1}(\curl, \Omega)=(\boldsymbol H_0(\curl , \Omega))^{\prime}=\boldsymbol H^{-1}({\rm div}, \Omega),
\end{equation}
as required.
\end{proof}

Then we construct the commutative diagram
\begin{equation}\label{cd:Hm1curl}
\begin{array}{c}
\xymatrix{
L_0^2(\Omega) \ar[r]^-{\grad} & \boldsymbol H^{-1}({\curl}, \Omega) \ar[r]^-{\curl}
                & (K_0^c)' \ar[r]^-{} & 0 \\
&\boldsymbol H_0(\div , \Omega) \ar[u]^{I} & K_0^c \ar[u]_{\curl\curl}\ar[l]_-{\curl}}
\end{array}.
\end{equation}
Using Theorem~\ref{thm:abstracthelmholtzdecomp},
it holds the stable Helmholtz decomposition
\begin{equation}\label{eq:Hcurl_1helmholtzdecomp}
\boldsymbol H^{-1}({\curl}, \Omega) = \nabla L_0^2(\Omega) \oplus \curl K_0^c = \nabla L_0^2(\Omega) \oplus \curl \bs H_0(\curl, \Omega).
\end{equation}

According to Corollary~\ref{cor:dualcharac} and commutative diagram \eqref{cd:Hm1curl}, it follows
\[
(\boldsymbol H_0(\div , \Omega))'=\boldsymbol H^{-1}(\curl, \Omega),
\]
and thus the decomposition \eqref{eq:Hcurl_1helmholtzdecomp} is also orthogonal in $(\cdot,\cdot)_{H(\div)'}$ inner product.
A decomposition of the dual space of $\boldsymbol H_0^m(\div , \Omega):=\{\bs\phi\in\boldsymbol H_0(\div , \Omega): \div\bs\phi\in H_0^m(\Omega)\}$
was presented in \cite{AmroucheCiarletCiarlet2010}.

\subsection{Helmholtz decomposition of symmetric tensors: HHJ complex}\label{subsec:div2HD}
We now consider differential complexes involving symmetric tensor functions. 


\begin{lemma}
We have the following exact sequence
\begin{equation}\label{eq:hhjcomplex2d-0427}
\boldsymbol{H}^{1}(\Omega; \mathbb{R}^2)\;\autorightarrow{$\sym\bs\curl$}{} \; \boldsymbol{L}^{2}(\Omega; \mathbb{S})\; \autorightarrow{$\div \boldsymbol{\div }$}{} \;H^{-2}(\Omega)\;\autorightarrow{}{}\;0.
\end{equation}
\end{lemma}
\begin{proof}
The identity $\ker(\div \bs \div) = \img(\sym\bs\curl)$ can be found in \cite{HuangHuangXu2011}.
Thanks to the exact sequence \eqref{eq:deRhamcomplex2d-H} with $s=-1, -2$,
\[
\div\boldsymbol{\div}\boldsymbol{L}^{2}(\Omega; \mathbb{M})=\div\boldsymbol{H}^{-1}(\Omega; \mathbb{R}^2)=H^{-2}(\Omega).
\]
Noting that $\boldsymbol{L}^{2}(\Omega; \mathbb{M})=\boldsymbol{L}^{2}(\Omega; \mathbb{S})+\boldsymbol{L}^{2}(\Omega; \mathbb{K})$ and $\div\boldsymbol{\div}\boldsymbol{L}^{2}(\Omega; \mathbb{K})=0$, we achieve $\div\boldsymbol{\div}\boldsymbol{L}^{2}(\Omega; \mathbb{S})=\div\boldsymbol{\div}\boldsymbol{L}^{2}(\Omega; \mathbb{M})=H^{-2}(\Omega)$.
\end{proof}

With the exact sequence \eqref{eq:hhjcomplex2d-0427}
we construct the following commutative diagram
$$
\begin{array}{c}
\xymatrix{
\boldsymbol{H}^{1}(\Omega; \mathbb{R}^2) \ar[r]^-{\sym\bs\curl}
                & \boldsymbol{L}^{2}(\Omega; \mathbb{S}) \ar[r]^-{\div \boldsymbol{\div }} & H^{-2}(\Omega) \ar[r]^-{} & 0\\
&\boldsymbol{L}^{2}(\Omega; \mathbb{S}) \ar[u]^{I} & H_0^2(\Omega)\ar[u]_-{\Delta^{2}}\ar[l]_-{\boldsymbol \nabla^2}}
\end{array}.
$$
By Corollary~\ref{cor:orthdec},
we recover the $L^2$-orthogonal Helmholtz decomposition obtained in~\cite[Lemma~3.1]{HuangHuangXu2011}
$$
\boldsymbol{L}^{2}(\Omega; \mathbb{S}) = \sym\bs\curl \boldsymbol{H}^{1}(\Omega; \mathbb{R}^2) \oplus^{\bot} \boldsymbol\nabla^2 H_0^2(\Omega).
$$

We can follow the definition of $\bs H^{-1}(\div,\Omega)$ to introduce the following spaces
\[
\boldsymbol{H}^{-1}(\div \boldsymbol{\div },\Omega; \mathbb{S}):=\{\boldsymbol{\tau}\in \boldsymbol{L}^{2}(\Omega; \mathbb{S}): \div \mathbf{div}\boldsymbol{\tau}\in H^{-1}(\Omega)\}
\]
with squared norm $
\|\boldsymbol{\tau}\|_{\boldsymbol{H}^{-1}(\div \boldsymbol{\div })}^2:=\|\boldsymbol{\tau}\|_{0}^2+\|\div \boldsymbol{\div }\boldsymbol{\tau}\|_{-1}^2
$, and
\[
\boldsymbol{H}(\div \boldsymbol{\div },\Omega; \mathbb{S}):=\{\boldsymbol{\tau}\in \boldsymbol{L}^{2}(\Omega; \mathbb{S}): \div \mathbf{div}\boldsymbol{\tau}\in L^{2}(\Omega)\}
\]
with squared norm $
\|\boldsymbol{\tau}\|_{\boldsymbol{H}(\div \boldsymbol{\div })}^2:=\|\boldsymbol{\tau}\|_{0}^2+\|\div \boldsymbol{\div }\boldsymbol{\tau}\|_{0}^2
$.
We recall the Hellan-Herrmann-Johnson (HHJ) exact sequence (cf.~\cite[Lemma~2.2]{ChenHuHuang2015a}) and give a simple proof here.
\begin{lemma}
We have the following exact sequence
\begin{equation}\label{eq:hhjcomplex2d}
\boldsymbol{H}^{1}(\Omega; \mathbb{R}^2)\;\autorightarrow{$\sym\bs\curl$}{} \; \boldsymbol{H}^{-1}(\div \boldsymbol{\div },\Omega; \mathbb{S})\; \autorightarrow{$\div \boldsymbol{\div }$}{} \;H^{-1}(\Omega) \; \autorightarrow{}{}\; 0.
\end{equation}
\end{lemma}
\begin{proof}
We only need to prove $\div\bs \div\boldsymbol{H}^{-1}(\div \boldsymbol{\div},\Omega; \mathbb{S})=H^{-1}(\Omega)$.
By the definition of $\boldsymbol{H}^{-1}(\div \boldsymbol{\div },\Omega; \mathbb{S})$, apparently $\div\bs \div\boldsymbol{H}^{-1}(\div \boldsymbol{\div },\Omega; \mathbb{S})\subset H^{-1}(\Omega)$. On the other side, for each $v\in H^{-1}(\Omega)\subset H^{-2}(\Omega)$, by the exact sequence \eqref{eq:hhjcomplex2d-0427}, there exists $\bs\tau\in\boldsymbol{L}^{2}(\Omega; \mathbb{S})$ such that $\div \boldsymbol{\div}\bs\tau=v$.
Note that $v\in H^{-1}(\Omega)$, thus $\bs\tau\in\boldsymbol{H}^{-1}(\div \boldsymbol{\div },\Omega; \mathbb{S})$, which indicates $H^{-1}(\Omega)\subset\div\bs \div\boldsymbol{H}^{-1}(\div \boldsymbol{\div},\Omega; \mathbb{S})$.
\end{proof}

Given a scalar function $v$, we can embed it into the symmetric tensor space as $\boldsymbol\pi (v) = v \boldsymbol I_{2\times 2}$. Since $\Delta v = \div \boldsymbol{\div } \boldsymbol\pi (v)$, we have the commutative diagram in two dimensions 
\begin{equation}\label{cd:hhj}
\begin{array}{c}
\xymatrix{
\boldsymbol{H}^{1}(\Omega; \mathbb{R}^2) \ar[r]^-{\sym\bs\curl}
                & \boldsymbol{H}^{-1}(\div \boldsymbol{\div },\Omega; \mathbb{S}) \ar[r]^-{\div \boldsymbol{\div }} & H^{-1}(\Omega) \ar[r]^-{} & 0\\
&\boldsymbol H_0^1(\Omega; \mathbb S) \ar[u]^{I} & H_0^1(\Omega)\ar[u]_{\Delta}\ar[l]_{\boldsymbol\pi}}
\end{array}.
\end{equation}
According to Theorem~\ref{thm:abstracthelmholtzdecomp},
we recover the recent Hemholtz decomposition in~\cite[Theorem 3.1]{KrendlRafetsederZulehner2016}
\begin{equation}\label{Hdivdiv_1symdec}
\boldsymbol{H}^{-1}(\div \boldsymbol{\div },\Omega; \mathbb{S}) = \sym\bs\curl \boldsymbol{H}^{1}(\Omega; \mathbb{R}^2) \oplus \boldsymbol\pi H_0^1(\Omega).
\end{equation}
The index $-1$ can be further decreased to $-2$; see Lemma \ref{lm:H-2elasticity} in the next subsection.
The generalized Hemholtz decomposition~\eqref{Hdivdiv_1symdec} is not orthogonal in $\boldsymbol H^{-1}(\Omega; \mathbb S)$ or $\boldsymbol{H}^{-1}(\div \boldsymbol{\div },\Omega; \mathbb{S})$ inner product. Indeed $\boldsymbol\pi\neq(\div \boldsymbol{\div })'$,
thus \eqref{Hdivdiv_1symdec} is derived from Theorem~\ref{thm:abstracthelmholtzdecomp} for Banach spaces rather than Corollary~\ref{cor:orthdec} for Hilbert spaces.

Smoothness of the symmetric tensor can be further increased to
$$
\begin{array}{c}
\xymatrix{
\boldsymbol{H}^{1}(\Omega; \mathbb{R}^2) \ar[r]^-{\sym\bs\curl}
                & \boldsymbol{H}(\div \boldsymbol{\div },\Omega; \mathbb{S}) \ar[r]^-{\div \boldsymbol{\div }} & L^{2}(\Omega) \ar[r]^-{} & 0\\
&\bs H_0^1(\Omega; \mathbb S) \ar[u]^{I} & L^2(\Omega)\ar[u]_{I}\ar[l]_-{\boldsymbol\pi\Delta^{-1}}}
\end{array},
$$
which leads to the generalized Helmholtz decomposition
$$
\boldsymbol{H}(\div \boldsymbol{\div },\Omega; \mathbb{S}) = \sym\bs\curl \boldsymbol{H}^{1}(\Omega; \mathbb{R}^2) \oplus \boldsymbol\pi \Delta^{-1}L^2(\Omega).
$$

\subsection{Helmholtz decomposition of symmetric tensors: linear elasticity}\label{subsec:rot2HD}

Recall that the symmetric gradient $\bs \varepsilon(\bs u) = (\nabla u + (\nabla u)^{\intercal})/2$.
Let
\[
\boldsymbol{H}^{-2}(\rot \mathbf{rot},\Omega; \mathbb{S}):=\{\boldsymbol{\tau}\in \boldsymbol{H}^{-1}(\Omega; \mathbb{S}): \rot \mathbf{rot}\boldsymbol{\tau}\in H^{-2}(\Omega)\}
\]
with squared norm
$
\|\boldsymbol{\tau}\|_{\boldsymbol{H}^{-2}(\rot \mathbf{rot})}^2:=\|\boldsymbol{\tau}\|_{-1}^2+\|\rot \mathbf{rot}\boldsymbol{\tau}\|_{-2}^2.
$

\begin{lemma}\label{lm:H-2elasticity}
The complex
\begin{equation}\label{complex:2-7}
\boldsymbol{L}^{2}(\Omega; \mathbb{R}^2)\;\autorightarrow{$\boldsymbol \varepsilon$}{} \; \boldsymbol{H}^{-2}(\rot \mathbf{rot},\Omega; \mathbb{S})\; \autorightarrow{$\rot \boldsymbol{\rot }$}{} \;H^{-2}(\Omega)\;\autorightarrow{}{}\;0
\end{equation}
is exact.
\end{lemma}
\begin{proof}
It is trivial that \eqref{complex:2-7} is a complex, i.e., $\rot \bs \rot \circ \bs \varepsilon = 0$. Next we show the exactness.
For any $\bs\tau\in \ker(\rot\boldsymbol{\rot})$, by the exact sequence \eqref{eq:deRhamcomplex2d-H-rot} with $s=-3$, there exists $v\in H^{-1}(\Omega)$ satisfying $\boldsymbol{\rot}\bs\tau=\nabla v$. Since $\nabla v=\boldsymbol{\rot}\begin{pmatrix}0 & v \\ -v & 0\end{pmatrix}$, we have
\[
\boldsymbol{\rot}\left(\bs\tau-\begin{pmatrix}0 & v \\ -v & 0\end{pmatrix}\right)=\bs0.
\]
Thus $\bs\tau=\begin{pmatrix}0 & v \\ -v & 0\end{pmatrix}+\bs\nabla\bs \phi$ with $\bs \phi\in\boldsymbol{L}^{2}(\Omega; \mathbb{R}^2)$, which together with the fact that $\bs\tau$ is symmetric means $\bs\tau=\boldsymbol \varepsilon(\bs\phi)$. Hence $\ker(\rot\boldsymbol{\rot})\subset\img(\bs\varepsilon)$.

By the rotation of the exact sequence~\eqref{eq:hhjcomplex2d-0427},
$
\rot\boldsymbol{\rot}\boldsymbol{L}^{2}(\Omega; \mathbb{S})=H^{-2}(\Omega).
$
Thus
$
H^{-2}(\Omega)\subset \rot\boldsymbol{\rot}\boldsymbol{H}^{-2}(\rot \mathbf{rot},\Omega; \mathbb{S}),
$
which implies $\img(\rot\bs\rot)=H^{-2}(\Omega)$.
\end{proof}

With the complex \eqref{complex:2-7} and the fact that $\Delta^2 = \rot \mathbf{rot}\boldsymbol I\bs\curl\curl $,
we have the commutative diagram
\begin{equation}\label{cd:Rotrot_2symdec}
\begin{array}{c}
\xymatrix{
\boldsymbol L^{2}(\Omega; \mathbb R^2) \ar[r]^-{\boldsymbol \varepsilon}
                &\boldsymbol H^{-2}(\rot \mathbf{rot}, \Omega;\mathbb S) \ar[r]^-{\rot \mathbf{rot}} & H^{-2}(\Omega) \ar[r]^-{} & 0\\
&\boldsymbol H_0(\div , \Omega;\mathbb S) \ar[u]^{\boldsymbol I} & H_0^2(\Omega)\ar[u]_{\Delta^2}\ar[l]_-{\bs\curl\curl }}
\end{array}
\end{equation}
which leads to a Helmholtz decomposition
\begin{equation}\label{Rotrot_2symdec}
\boldsymbol H^{-2}(\rot \mathbf{rot}, \Omega;\mathbb S) = \boldsymbol \varepsilon \boldsymbol L^2(\Omega; \mathbb R^2) \oplus \bs\curl\curl  H_0^2(\Omega).
\end{equation}

Applying Corollary~\ref{cor:dualcharac} to commutative diagram \eqref{cd:Rotrot_2symdec}, we have
\[
(\boldsymbol H_0(\div , \Omega;\mathbb S))'=\boldsymbol H^{-2}(\rot \mathbf{rot}, \Omega;\mathbb S).
\]
Therefore the decomposition \eqref{Rotrot_2symdec} is also orthogonal in $\boldsymbol H_0(\div , \Omega;\mathbb S))'$ inner product.

\subsection{Helmholtz decomposition of $\boldsymbol{H}(\curl\curl, (K_0^c)')$ space}\label{subsec:curl2KdualHD}
Let $$\boldsymbol{L}_0^2(\Omega; \mathbb M):=\{\bs\tau\in\boldsymbol{L}^2(\Omega; \mathbb M): (\tr\bs\tau, 1)=0\}.$$
Introduce two Hilbert spaces
\[
\boldsymbol{H}^0(\bs\div, \dev\sym):=\{\bs\tau\in\boldsymbol{L}_0^2(\Omega; \mathbb M): \div\bs\tau=\bs 0, \dev\sym\bs\tau=\bs 0\}
\]
with norm $\|\cdot\|_0$, and
\[
\boldsymbol{H}(\curl\curl, (K_0^c)'):=\{\bs\phi\in\boldsymbol L^2(\Omega; \mathbb R^3): \curl\curl\bs\phi\in(K_0^c)'\}
\]
with squared norm
$
\|\boldsymbol{\tau}\|_{\boldsymbol{H}(\curl\curl, (K_0^c)')}^2:=\|\boldsymbol{\tau}\|_{0}^2+\|\curl\curl\boldsymbol{\tau}\|_{(K_0^c)'}^2
$
.
\begin{lemma}\label{lem:h0divdevsym}
The complex
\[
\boldsymbol{H}^0(\bs\div, \dev\sym)\;\autorightarrow{$\spn^{-1}\skw$}{} \; \boldsymbol{H}(\curl\curl, (K_0^c)')\; \autorightarrow{$\curl\curl$}{} \;(K_0^c)'\;\autorightarrow{}{}\;0
\]
is exact.
\end{lemma}
\begin{proof}
By definition $\img(\curl\curl)=(K_0^c)'$.
Next we show that $\ker(\curl\curl)=\img(\spn^{-1}\skw)$.

Taking any $\bs\tau\in\boldsymbol{H}^0(\bs\div, \dev\sym)$, set $\bs\psi=\spn^{-1}\skw\bs\tau$. Then we have
$\spn\bs\psi=\skw\bs\tau$. Noting that $\dev\sym\bs\tau=\bs 0$, i.e. $\sym\bs\tau=\frac{1}{3}(\tr\bs\tau)\bs I$, it holds
\[
\bs\tau=\sym\bs\tau+\skw\bs\tau=\frac{1}{3}(\tr\bs\tau)\bs I+\spn\bs\psi.
\]
Hence by $\div\bs\tau=\bs 0$ we obtain $\frac{1}{3}\nabla(\tr\bs\tau)=-\bs\div\spn\bs\psi=\curl\bs\psi$, which implies $\curl\curl\bs\psi=\bs0$.
Thus we have $\img(\spn^{-1}\skw)\subset\ker(\curl\curl)$.

On the other hand, take any $\bs\psi\in\ker(\curl\curl)$. Then there exists $w\in L_0^2(\Omega)$ such that $\curl\bs\psi=\nabla w$.
Let
\begin{equation}\label{eq:temp20180118-1}
\bs\tau=\spn\bs\psi+w\bs I.
\end{equation}
Obviously $\bs\div\bs\tau=0$,  $w\bs I=\sym\bs\tau$ and $w=\frac{1}{3}\tr(\sym\bs\tau)$.
Thus $\dev\sym\bs\tau=\bs 0$, i.e. $\bs\tau\in\boldsymbol{H}^0(\bs\div, \dev\sym)$.
By applying $\skw$ to \eqref{eq:temp20180118-1}, we have $\spn\bs\psi=\skw\bs\tau$, which indicates $\bs\psi=\spn^{-1}\skw\bs\tau\in \img(\spn^{-1}\skw)$.
\end{proof}

According to the proof of Lemma~\ref{lem:h0divdevsym}, we have
\begin{equation}\label{eq:H0divdevsym}
\boldsymbol{H}^0(\bs\div, \dev\sym)=\{\bs\tau=v\bs I+\spn\bs\phi: v\in L_0^2(\Omega), \bs\phi\in\boldsymbol{L}^2(\Omega; \mathbb R^3), \div\bs\tau=0\}.
\end{equation}
Therefore we have the commutative diagram
\begin{equation}\label{cd:curlcurldual}
\begin{array}{c}
\xymatrix@R-=1.0pc@C=1.28cm{
\boldsymbol{H}^0(\bs\div, \dev\sym) \ar[r]^-{\spn^{-1}\skw} & \boldsymbol{H}(\curl\curl, (K_0^c)') \ar[r]^-{\curl\curl}
                & (K_0^c)' \ar[r]^-{} & 0 \\
&K_0^c  \ar[u]^{I} & K_0^c\ar[u]_{\curl\curl}\ar[l]_{I}}
\end{array}.
\end{equation}
Applying Theorem~\ref{thm:abstracthelmholtzdecomp}, 
we get the stable Helmholtz decomposition
\begin{equation}\label{complex:curlcurldual}
\boldsymbol{H}(\curl\curl, (K_0^c)')=\spn^{-1}\skw\boldsymbol{H}^0(\bs\div, \dev\sym)\oplus K_0^c.
\end{equation}
Again this Helmholtz decomposition is not orthogonal in $\boldsymbol{H}(\curl\curl, (K_0^c)')$ or $(K_0^c)'$ inner product.

\subsection{Helmholtz decomposition of $\boldsymbol{H}^{-2}(\div^3, \Omega)$ space}\label{subsec:div3HD}
Denote
\[
\boldsymbol{H}^{-2}(\div^3, \Omega):=\{\bs\tau\in\boldsymbol{L}^2(\Omega; \mathbb S(3)): \div\bs\div\bs\div\bs\tau\in H^{-2}(\Omega)\}
\]
with squared norm
$
\|\boldsymbol{\tau}\|_{\boldsymbol{H}^{-2}(\div^3)}^2:=\|\boldsymbol{\tau}\|_{0}^2+\|\div\bs\div\bs\div\bs\tau\|_{-2}^2
$.
Define $\bs\Xi: H_0^1(\Omega; \mathbb R^2)\to \boldsymbol{H}^{-2}(\div^3, \Omega)$ as follows: for any $\bs\psi=(\psi_1, \psi_2)^{\intercal}\in H_0^1(\Omega; \mathbb R^2)$,
$\bs\Xi\bs\psi:=(\bs\tau_{ijk})_{2\times2\times2}$ with
\[
\bs\tau_{111}=\psi_1, \; \bs\tau_{222}=\psi_2, \; \bs\tau_{112}=\bs\tau_{121}=\bs\tau_{211}=\frac{1}{3}\psi_2, \;\bs\tau_{122}=\bs\tau_{212}=\bs\tau_{221}=\frac{1}{3}\psi_1.
\]

\begin{lemma}
The complex
\[
\boldsymbol{H}^1(\Omega; \mathbb S)\;\autorightarrow{$\sym\bs\curl$}{} \; \boldsymbol{H}^{-2}(\div^3, \Omega)\; \autorightarrow{$\div\bs\div\bs\div$}{} \;H^{-2}(\Omega)\;\autorightarrow{}{}\;0
\]
is exact.
\end{lemma}
\begin{proof}
It is apparent that $\Delta^2=\div\bs\div\bs\div\bs\Xi\nabla$, thus $\img(\div\bs\div\bs\div)=H^{-2}(\Omega)$.
We refer to \cite[Theorem 4.1 and Theorem 4.4]{Schedensack2016} for $\ker(\div\bs\div\bs\div)=\img(\sym\bs\curl)$.
\end{proof}

Through constructing the commutative diagram
\begin{equation}\label{cd:div3m2}
\begin{array}{c}
\xymatrix@R-=1.0pc@C=1.6cm{
\boldsymbol{H}^1(\Omega; \mathbb S) \ar[r]^-{\sym\bs\curl} & \boldsymbol{H}^{-2}(\div^3, \Omega) \ar[r]^-{\div\bs\div\bs\div}
                & H^{-2}(\Omega) \ar[r]^-{} & 0 \\
&H_0^1(\Omega; \mathbb R^2)  \ar[u]^{\bs\Xi} & H_0^{2}(\Omega)\ar[u]_{\Delta^2}\ar[l]_{\nabla}}
\end{array},
\end{equation}
we get the stable Helmholtz decomposition from Theorem~\ref{thm:abstracthelmholtzdecomp}
\begin{equation}\label{complex:div3m2}
\boldsymbol{H}^{-2}(\div^3, \Omega)=\sym\bs\curl\boldsymbol{H}^1(\Omega; \mathbb S)\oplus\bs\Xi\nabla H_0^{2}(\Omega).
\end{equation}

More differential complexes and Helmholtz decompositions can be obtained and some of them will be discussed along with the mixed formulations of elliptic systems.

\section{Abstract Mixed Formulation and Its Decomposition}
In this section we present an abstract mixed formulation and use a Helmholtz decomposition to decouple the saddle point system into several elliptic problems.

\subsection{Framework}
Assume we have the exact sequence
\begin{equation}\label{eq:shortcomplex0427}
P\; \autorightarrow{$\dd ^-$}{}\; \Sigma \; \autorightarrow{$\dd$}{} \; V' \; \autorightarrow{}{}\; 0,
\end{equation}
and
the commutative diagram
\begin{equation}\label{eq:shortdiagrammixedform}
\begin{array}{c}
\xymatrix@R-=1.6pc{
X \ar@{->}[r]^-{J_X} & X' & \\
P \ar[r]^-{\dd ^-} & \Sigma \ar[r]^-{\dd}\ar@{}[u]|{\bigcup}
                & V'\ar[r]^-{} & 0\\
&\widetilde{\Sigma}  \ar[u]^{\Pi_{\Sigma}} & V\ar[u]_{J_V}\ar[l]_{\Pi_V}}
\end{array},
\end{equation}
where the isomorphisms $J_X$ and $J_V$ are given by \eqref{eq:J_X}, i.e., the inverse of the Riesz representation operator, and the rest linear operators are all continuous but not necessarily isomorphic.

By Theorem~\ref{thm:abstracthelmholtzdecomp}, we have a stable Helmholtz decomposition
\begin{equation}\label{eq:abstracthelmholtzdecomp2}
\Sigma = \dd ^-P \oplus \Pi_{\Sigma} \Pi_V V.
\end{equation}
We emphasize that we do not need to know neither the short exact sequence at the bottom nor the space $\widetilde{\Sigma}$ in a very precise form (i.e. $\widetilde{\Sigma}$ can be reasonably enlarged to include the image space $\Pi_V V$).

\subsubsection{Two-term formulation}
An abstract mixed formulation based on the commutative diagram~\eqref{eq:shortdiagrammixedform} is: given $g\in\Sigma'$ and $f\in V'$,  find $(\sigma , u)\in \Sigma\times V$ such that
\begin{align}
(\sigma, \tau)_{X'}+ \langle \dd \tau, u\rangle & =\langle g, \tau\rangle \quad \forall~\tau\in\Sigma, \label{eq:mixed1}\\
\langle \dd \sigma, v\rangle & =\langle f, v\rangle  \quad \forall~v\in V.\label{eq:mixed2}
\end{align}

The bilinear form $(\cdot, \cdot)_{X'}$ is not necessary to be an inner product unless we intend to involve $X$ and $J_X$ in the mixed formulation.
We only require $(\cdot, \cdot)_{X'}$ is positive semidefinite and symmetric.

To show the well-posedness of the mixed formulation \eqref{eq:mixed1}-\eqref{eq:mixed2}, we assume the following norm equivalence
\begin{equation}\label{eq:normequivalenceSigma}
\|\tau\|_{\Sigma}^2\eqsim \|\tau\|_{X'}^2 + \|\dd \tau\|_{V'}^2 \quad \forall~\tau\in\Sigma.
\end{equation}
This norm equivalence is usually trivial, and it holds apparently for all the examples in this paper. Indeed the space $\Sigma$ is usually constructed so that \eqref{eq:normequivalenceSigma} holds.

\begin{theorem}\label{thm:mixedwellposed}
Assume the exact sequence \eqref{eq:shortcomplex0427}, the commutative diagram \eqref{eq:shortdiagrammixedform} and the norm equivalence \eqref{eq:normequivalenceSigma} hold, then
the mixed formulation \eqref{eq:mixed1}-\eqref{eq:mixed2} is uniquely solvable. Moreover, we have the stability result
\[
\|\sigma\|_{\Sigma} + \|u\|_V\lesssim \|g\|_{\Sigma'}+\|f\|_{V'}.
\]
\end{theorem}
\begin{proof}
It is trivial that the bilinear forms in the mixed formulation \eqref{eq:mixed1}-\eqref{eq:mixed2} are continuous due to \eqref{eq:normequivalenceSigma}.
Using \eqref{eq:normequivalenceSigma} again, it is also obvious that
\[
\|\tau\|_{\Sigma}\lesssim \|\tau\|_{X'} + \|\dd \tau\|_{V'}=\|\tau\|_{X'} \quad \forall~\tau\in\ker d.
\]
By Babu\v{s}ka-Brezzi theory (cf.~\cite{BabuvskaAziz1972, Brezzi1974, BoffiBrezziFortin2013}), it suffices to prove the inf-sup condition
\begin{equation}\label{eq:infsup}
\|v\|_V\lesssim \sup_{\tau\in\Sigma}\frac{\langle \dd \tau, v\rangle}{\|\tau\|_{\Sigma}} \quad \forall~v\in V.
\end{equation}
For each $v\in V$, let $\tau=\Pi_{\Sigma} \Pi_V v$. It is apparent that
\[
\|\tau\|_{\Sigma}=\|\Pi_{\Sigma} \Pi_V v\|_{\Sigma}\lesssim \|v\|_V.
\]
Then making use of the commutative diagram \eqref{eq:shortdiagrammixedform} and \eqref{eq:J_X}, it follows
\[
\langle \dd \tau, v\rangle=\langle \dd \Pi_{\Sigma} \Pi_V v, v\rangle=\langle J_V v, v\rangle=\|v\|_V^2.
\]
Hence we have
\[
\|v\|_V\|\tau\|_{\Sigma}\lesssim \|v\|_V^2=\langle \dd \tau, v\rangle,
\]
which means the inf-sup condition \eqref{eq:infsup}.
\end{proof}

\subsubsection{Unfolded three-term formulation}
We derive an equivalent three-term formulation of the mixed formulation \eqref{eq:mixed1}-\eqref{eq:mixed2} when $X'$ is a Sobolev space of negative order. In this subsection, we assume the bilinear form $(\cdot, \cdot)_{X'}$ is the corresponding inner product of $X'$.

Let $\phi=J_X^{-1}\sigma\in X$. By \eqref{eq:Xdualinnerproduct}, we can rewrite \eqref{eq:mixed1} as
\[
\langle\tau, \phi\rangle+ \langle \dd \tau, u\rangle =\langle g, \tau\rangle \quad \forall~\tau\in\Sigma.
\]
Noting that $\sigma=J_X\phi$, it follows from \eqref{eq:J_X}
\[
\langle\sigma, \psi\rangle=\langle J_X\phi, \psi\rangle=(\phi, \psi)_X \quad \forall~\psi\in X.
\]
Therefore the mixed formulation \eqref{eq:mixed1}-\eqref{eq:mixed2} is equivalent to an unfolded three-term formulation:
find $(\phi, u, \sigma)\in X\times V\times\Sigma$ such that
\begin{align}
(\phi, \psi)_X - \langle \sigma, \dd 'v + \psi \rangle & =-\langle f, v\rangle + \langle g_X, \psi\rangle  \quad \forall~(\psi, v)\in X\times V, \label{eq:mixedunfolded1} \\
\langle \dd 'u + \phi , \tau \rangle & =  \langle g, \tau\rangle \quad\quad\quad\quad\quad\quad\, \forall~\tau\in\Sigma, \label{eq:mixedunfolded2}
\end{align}
with $g_X=0$.
It is interesting to note that the variable $\sigma$ can be formally interpreted as the Lagrange multiplier to impose the constraint $I'\phi = -\dd ' u$ in $\Sigma '$ if $g=0$, where $I: \Sigma\to X'$ is the natural embedding operator.
The operator equation of the mixed formulation \eqref{eq:mixedunfolded1}-\eqref{eq:mixedunfolded2} is
\begin{equation}\label{eq:threeoperator}
\begin{pmatrix}
J_X & 0 & -I \\
0 & 0 & -\dd \\
I' & \dd' & 0
\end{pmatrix}
\begin{pmatrix}
\phi \\
u\\
\sigma
\end{pmatrix}
=\begin{pmatrix}
g_X \\
-f\\
g
\end{pmatrix}.
\end{equation}

According to Theorem~\ref{thm:mixedwellposed}, we immediately obtain the well-posedness of the mixed formulation \eqref{eq:mixedunfolded1}-\eqref{eq:mixedunfolded2} for $g_X=0$.
The well-posedness of the mixed formulation \eqref{eq:mixedunfolded1}-\eqref{eq:mixedunfolded2} for general $g_X\in X'$ is given as follows.
\begin{theorem}\label{thm:mixedunfoldedwellposed}
Assume the exact sequence \eqref{eq:shortcomplex0427}, the commutative diagram \eqref{eq:shortdiagrammixedform} and the norm equivalence \eqref{eq:normequivalenceSigma} hold, then
the unfolded mixed formulation \eqref{eq:mixedunfolded1}-\eqref{eq:mixedunfolded2} is uniquely solvable. Moreover, we have the stability result
\[
\|\phi\|_{X} + \|\sigma\|_{\Sigma} + \|u\|_V\lesssim \|g_X\|_{X'}+\|g\|_{\Sigma'}+\|f\|_{V'}.
\]
\end{theorem}
\begin{proof}
We split \eqref{eq:threeoperator} as a $2\times 2$ block saddle point system by treating $X\times V$ as one space and denote by $B = [I' \; \dd ']: X\times V \to \Sigma'$ as $B((\phi, u)) = I'\phi + \dd' u$.

By \eqref{eq:normequivalenceSigma}, all the bilinear forms in the mixed formulation \eqref{eq:mixedunfolded1}-\eqref{eq:mixedunfolded2} are obviously continuous.
Let $(\psi, v)\in \ker(B)\cap (X\times V)$, i.e., $I'\psi=-\dd 'v$.
Due to \eqref{eq:J_X} and the commutative diagram \eqref{eq:shortdiagrammixedform}, we get
\begin{align*}
\|v\|_V^2=&\langle J_Vv, v\rangle=\langle \dd \Pi_{\Sigma} \Pi_V v, v\rangle=\langle \Pi_{\Sigma} \Pi_V v, \dd 'v\rangle \\
\leq &\|\Pi_{\Sigma} \Pi_V v\|_{\Sigma}\|\dd 'v\|_{\Sigma'}\lesssim \|v\|_V\|\dd 'v\|_{\Sigma'}.
\end{align*}
Noting that $I': X\to \Sigma'$ is continuous, we have
\[
\|v\|_V\lesssim \|\dd 'v\|_{\Sigma'}=\|I'\psi\|_{\Sigma'}\lesssim \|\psi\|_{X},
\]
which implies the coercivity on the kernel of $B$.

On the other hand, for any $\tau\in\Sigma$ it follows from \eqref{eq:normequivalenceSigma} that
\[
\|\tau\|_{\Sigma}\lesssim \|\tau\|_{X'} + \|\dd \tau\|_{V'}=\sup_{\psi\in X}\frac{\langle\tau, I'\psi\rangle}{\|\psi\|_X} + \sup_{v\in V}\frac{\langle \dd \tau, v\rangle}{\|v\|_V}\lesssim \sup_{\psi\in X, v\in V}\frac{\langle\tau, I'\psi + \dd' v\rangle}{\|\psi\|_X+ \|v\|_V},
\]
which is just the inf-sup condition of $B$. Therefore the required result is guaranteed by Babu\v{s}ka-Brezzi theory.
\end{proof}


\subsubsection{Decoupled formulation}
We decompose the mixed formulation \eqref{eq:mixed1}-\eqref{eq:mixed2} using the Helmholtz decomposition \eqref{eq:abstracthelmholtzdecomp2}.
Applying the Helmholtz decomposition~\eqref{eq:abstracthelmholtzdecomp2} to both the trial and test functions
\[
\sigma = \dd ^-p +\Pi_{\Sigma} \Pi_V w, \quad \tau = \dd ^-q +\Pi_{\Sigma} \Pi_V \chi,
\]
where $p, q\in P/\ker \dd ^-$, and $w, \chi\in V$. Then substituting them into the mixed formulation \eqref{eq:mixed1}-\eqref{eq:mixed2}, we have
\begin{align}
(\dd ^-p +\Pi_{\Sigma} \Pi_V w, \Pi_{\Sigma} \Pi_V \chi)_{X'}+ \langle \dd \Pi_{\Sigma} \Pi_V \chi, u\rangle&=\langle g, \Pi_{\Sigma} \Pi_V \chi\rangle, \label{eq:mixedtemp1}\\
(\dd ^-p +\Pi_{\Sigma} \Pi_V w, \dd ^-q)_{X'} &=\langle g, \dd ^-q\rangle, \label{eq:mixedtemp2}\\
\langle \dd \Pi_{\Sigma} \Pi_V w, v\rangle&=\langle f, v\rangle, \label{eq:mixedtemp3}
\end{align}
for any $\chi\in V$, $q\in P/\ker \dd ^-$ and $v\in V$.
We obtain from the commutative diagram \eqref{eq:shortdiagrammixedform} and \eqref{eq:J_X} again
\[
\langle \dd \Pi_{\Sigma} \Pi_V w, v\rangle=\langle J_V w, v\rangle = ( w, v)_V, \quad \langle \dd \Pi_{\Sigma} \Pi_V \chi, u\rangle=(\chi, u)_V.
\]
Therefore, the mixed formulation \eqref{eq:mixedtemp1}-\eqref{eq:mixedtemp3} is equivalent to (in backwards order): find $w\in V$, $p\in P/\ker \dd ^-$, and $u\in V$ such that
\begin{align}
(w, v)_V&=\langle f, v\rangle  \quad\quad\quad\quad\quad\quad\quad\quad\quad\quad\quad \forall~v\in V, \label{eq:mixedequiv1} \\
(\dd ^-p, \dd ^-q)_{X'} &=\langle g, \dd ^-q\rangle - (\Pi_{\Sigma} \Pi_V w, \dd ^-q)_{X'} \quad\; \forall~q\in P/\ker \dd ^-, \label{eq:mixedequiv2} \\
(u, \chi)_V&=\langle g, \Pi_{\Sigma} \Pi_V \chi\rangle - (\sigma, \Pi_{\Sigma} \Pi_V \chi)_{X'} \;\;\; \forall~\chi\in V, \label{eq:mixedequiv3}
\end{align}
where $\sigma = \dd ^-p +\Pi_{\Sigma} \Pi_V w$.

\begin{remark}\rm
When the decomposition \eqref{eq:abstracthelmholtzdecomp2} is orthogonal with respect to $(\cdot, \cdot)_{X'}$ and $g=0$, the second equation~\eqref{eq:mixedequiv2} will disappear. $\hfill\Box$
\end{remark}

Applying the Helmholtz decomposition \eqref{eq:abstracthelmholtzdecomp2} to the unfolded formulation, the decoupled mixed formulation \eqref{eq:mixedunfolded1}-\eqref{eq:mixedunfolded2} is equivalent to find $w, u\in V$, $\phi\in X$, and $p\in P/\ker \dd ^-$ such that
\begin{align}
(w, v)_V&=\langle f, v\rangle  \quad\quad\quad\quad\quad\;\;\; \forall~v\in V, \label{eq:mixedunfoldedequiv1} \\
(\phi, \psi)_{X} - \langle \dd ^-p, \psi\rangle &= \langle\Pi_{\Sigma} \Pi_V w, \psi\rangle \quad\quad\;\; \forall~\psi\in X, \label{eq:mixedunfoldedequiv2} \\
\langle \dd ^-q, \phi\rangle &=\langle g, \dd ^-q\rangle \quad\quad\quad\quad\;\; \forall~q\in P/\ker \dd ^-, \label{eq:mixedunfoldedequiv3} \\
(u, \chi)_V&=\langle g-\phi, \Pi_{\Sigma} \Pi_V \chi\rangle \;\;\;\, \forall~\chi\in V. \label{eq:mixedunfoldedequiv4}
\end{align}
The middle system \eqref{eq:mixedunfoldedequiv2}-\eqref{eq:mixedunfoldedequiv3} of $(\phi, p)$ is now a Stokes-type system.

We summarize the former derivation as follows.
\begin{theorem}\label{thm:mixeddecouple}
Assume the exact sequence \eqref{eq:shortcomplex0427}, the commutative diagram \eqref{eq:shortdiagrammixedform} and the norm equivalence \eqref{eq:normequivalenceSigma} hold, then
the mixed formulation \eqref{eq:mixed1}-\eqref{eq:mixed2} can be decoupled as three elliptic equations \eqref{eq:mixedequiv1}-\eqref{eq:mixedequiv3} or four equations \eqref{eq:mixedunfoldedequiv1}-\eqref{eq:mixedunfoldedequiv4}.
\end{theorem}

\begin{remark}
Let $(\cdot, \cdot)_{\Sigma}$ be the inner product of Hilbert space $\Sigma$.
When
the space of harmonic forms
\[
\mathfrak H:=\left\{\tau\in\Sigma: \dd\tau=0,\; (\tau, \dd^- q)_{\Sigma}=0\quad \forall~q\in P\right\}
\]
is non-trivial,
by Remark~\ref{rmk:helmdecomHarmonic} we have
the Helmholtz decomposition
\begin{equation*}
\Sigma = \dd ^-P \oplus \Pi_{\Sigma} \Pi_V V \oplus \mathfrak H.
\end{equation*}
Then
the mixed formulation \eqref{eq:mixed1}-\eqref{eq:mixed2} can be decoupled as:
find $w, u\in V$, $p\in P/\ker \dd ^-$, and $r\in\mathfrak H$ such that
\begin{align*}
(w, v)_V&=\langle f, v\rangle  \quad\quad\quad\quad\quad\quad\quad\quad\quad\quad\quad\quad\quad\;\;\; \forall~v\in V, \\
(\dd ^-p+r, \dd ^-q+s)_{X'} &=\langle g, \dd ^-q+s\rangle - (\Pi_{\Sigma} \Pi_V w, \dd ^-q+s)_{X'} \; \forall~q\in P/\ker \dd ^-, s\in\mathfrak H, \\
(u, \chi)_V&=\langle g, \Pi_{\Sigma} \Pi_V \chi\rangle - (\sigma, \Pi_{\Sigma} \Pi_V \chi)_{X'} \quad\quad\quad\;\; \forall~\chi\in V,
\end{align*}
where $\sigma = \dd ^-p +\Pi_{\Sigma} \Pi_V w+r$. And similar modification can be applied to the decoupling of three-term formulation. In the decoupled formulation, however, the space of harmonic forms should be identified a priori which may not be easy for complicated geometric domains.
%
%
$\hfill\Box$
\end{remark}

In the rest of this section, we shall apply our abstract framework to several concrete examples.
It is worth mentioning again that the norm equivalence \eqref{eq:normequivalenceSigma} is trivial for all examples except the primal formulation of the fourth order $\curl$ equation. For this exceptional case, the norm equivalence \eqref{eq:normequivalenceSigma} is the result of \eqref{eq:temp20180205}.
Hence we will focus on the derivation of the commutative diagrams.

\subsection{The primal formulation of the biharmonic equation}\label{subsection:biharmonicdecouple}

Consider the biharmonic equation with homogenous Dirichlet boundary condition
\begin{equation}\label{eq:biharmonic}
\left\{
\begin{aligned}
&\Delta^2u=f\quad\quad\quad\;\;\text{in }\Omega,
\\
&u=\partial_{\nu}u=0\quad\quad\text{on }\partial\Omega,
\end{aligned}
\right.
\end{equation}
where $f\in H^{-1}(\Omega)$ and $\Omega\subset \mathbb{R}^n$ with $n=2, 3$.
The primal formulation of \eqref{eq:biharmonic} is to find $u\in H_0^2(\Omega)$ such that
\begin{equation}\label{eq:biharmonicprimal}
(\nabla^2u, \nabla^2 v)=\langle f, v\rangle\quad \forall~v\in H_0^2(\Omega).
\end{equation}

We first discuss the decoupling of \eqref{eq:biharmonicprimal} in two dimensions.
We recall
the known decoupling from \cite{HuangHuangXu2012,Huang2010, Gallistl2017}, which fits into the framework of this paper as
follows.
By the rotated version of the commutative diagram \eqref{Hdiv_1diagram}, we build up the commutative diagram
\begin{equation}\label{eq:biharmoniccommutativediagram}
\begin{array}{c}
\xymatrix@R-=1.0pc{
\bs H_0^{1}(\Omega; \mathbb R^2) \ar[r]^-{\Delta} & \boldsymbol{H}^{-1}(\Omega; \mathbb{R}^2) & \\
L_0^{2}(\Omega) \ar[r]^-{\grad} & \boldsymbol{H}^{-1}(\rot ,\Omega) \ar[r]^-{\rot }\ar@{}[u]|{\bigcup}
                & H^{-1}(\Omega) \ar[r]^-{} & 0 \\
&\boldsymbol H_0(\div, \Omega)  \ar[u]^-{I} & H_0^1(\Omega)\ar[u]_{\Delta}\ar[l]_-{\curl }}
\end{array},
\end{equation}
where recall that
\[
\boldsymbol H^{-1}(\rot , \Omega):=\{\boldsymbol{\phi}\in \boldsymbol{H}^{-1}(\Omega; \mathbb{R}^2): \rot \boldsymbol{\phi}\in H^{-1}(\Omega)\}
\]
with norm
$
\|\boldsymbol{\phi}\|_{H^{-1}(\rot )}^2:=\|\boldsymbol{\phi}\|_{-1}^2+\|\rot \boldsymbol{\phi}\|_{-1}^2.
$
According to Theorem~\ref{thm:abstracthelmholtzdecomp},
we have the Helmholtz decomposition
\begin{equation}\label{eq:Hrot_1helmholtzdecomp}
\boldsymbol{H}^{-1}(\rot ,\Omega) = \nabla L_0^{2}(\Omega) \oplus \curl  H_0^1(\Omega).
\end{equation}

The corresponding mixed formulation is to find $(\boldsymbol{\gamma} , u)\in \boldsymbol{H}^{-1}(\rot , \Omega)\times H_0^1(\Omega)$ such that
\begin{align}
(\boldsymbol\gamma, \boldsymbol\beta)_{-1}- \langle \rot \boldsymbol\beta, u\rangle & =0 \quad\quad\;\;\, \forall~\boldsymbol\beta\in\boldsymbol{H}^{-1}(\rot ,\Omega), \label{eq:biharmonicmixeda1}\\
\langle \rot \boldsymbol\gamma, v\rangle & =\langle f, v\rangle  \quad \forall~v\in H_0^1(\Omega), \label{eq:biharmonicmixeda2}
\end{align}
where $f\in H^{-1}(\Omega)$ and
$(\boldsymbol\gamma, \boldsymbol\beta)_{-1}:=-\langle\boldsymbol\Delta^{-1}\boldsymbol\gamma, \boldsymbol\beta\rangle=-\langle\boldsymbol\gamma, \boldsymbol\Delta^{-1}\boldsymbol\beta\rangle.$

By introducing variable
$\boldsymbol{\phi}=-\boldsymbol{\Delta}^{-1}\boldsymbol\gamma \in \boldsymbol{H}_0^1(\Omega; \mathbb{R}^2),$
the unfolded formulation is to find $(\boldsymbol{\gamma}, u, \boldsymbol{\phi})\in \boldsymbol{H}^{-1}(\rot , \Omega)\times H_0^1(\Omega)\times \boldsymbol H_0^1(\Omega; \mathbb{R}^2)$ such that
\begin{align}
\label{eq:biharmonic1} (\boldsymbol\nabla\boldsymbol\phi, \boldsymbol\nabla\boldsymbol\psi) + \langle \boldsymbol\gamma, \curl v-\boldsymbol\psi\rangle & =\langle f, v\rangle \quad \forall~(v, \boldsymbol{\psi})\in  H_0^1(\Omega)\times \boldsymbol H_0^1(\Omega; \mathbb{R}^2), \\
\label{eq:biharmonic2} \langle \boldsymbol\beta, \curl u-\boldsymbol\phi\rangle & =0  \quad\quad\;\;\; \forall~\boldsymbol\beta\in\boldsymbol{H}^{-1}(\rot ,\Omega),
\end{align}
which is just the rotation form of problem~(2.4) in~\cite{BrezziFortinStenberg1991}.

Equation \eqref{eq:biharmonic2} implies $\bs \phi = \curl u$, which is plugged into equation \eqref{eq:biharmonic1} by taking $\boldsymbol\psi=\curl v$ with $v\in H_0^2(\Omega)$ gives the primal formulation of the biharmonic equation \eqref{eq:biharmonicprimal} in two dimensions.

According to Theorem~\ref{thm:mixeddecouple}, the decoupled and unfolded formulation is to
find $w\in H_0^1(\Omega)$, $\boldsymbol{\phi}\in  \boldsymbol H_0^1(\Omega; \mathbb{R}^2)$, $p\in L_0^{2}(\Omega)$ and $u\in H_0^1(\Omega)$ such that
\begin{align}
(\curl  w, \curl  v)&=\langle f, v\rangle  \quad\quad\quad\quad \forall~v\in H_0^1(\Omega),  \label{eq:biharmonicdiscretedecouple1}\\
(\boldsymbol\nabla\boldsymbol\phi, \boldsymbol\nabla\boldsymbol\psi) + (\div\boldsymbol\psi, p) & =(\curl  w, \boldsymbol\psi) \quad\;\;\, \forall~\boldsymbol{\psi}\in  \boldsymbol H_0^1(\Omega; \mathbb{R}^2), \label{eq:biharmonicdiscretedecouple2}\\
(\div\boldsymbol\phi, q) &= 0 \quad\quad\quad\quad\quad\;\;\, \forall~ q\in L_0^{2}(\Omega),  \label{eq:biharmonicdiscretedecouple3}\\
(\curl  u, \curl  \chi)&= (\boldsymbol{\phi}, \curl \chi) \quad\quad \forall~\chi\in H_0^1(\Omega). \label{eq:biharmonicdiscretedecouple4}
\end{align}
Therefore we recover the decoupling that the primal formulation of the biharmonic equation \eqref{eq:biharmonicprimal} in two dimensions is equivalent to two Poisson equations and one Stokes equation~\cite{HuangHuangXu2012,Huang2010, Gallistl2017}.

Such decoupling of the biharmonic equation in two dimensions can be generalized in various ways.
First we consider
the fourth order elliptic singular perturbation problem with homogenous Dirichlet boundary condition
\begin{equation}\label{eq:4thsingular}
\varepsilon^2\Delta^2u-\Delta u=f.
\end{equation}
If we equip space $\boldsymbol H_0^1(\Omega; \mathbb{R}^2)$ with norm $(\|\cdot\|_0^2+\varepsilon^2|\cdot|_1^2)^{1/2}$ where $\varepsilon\geq0$,
the fourth order elliptic singular perturbation problem \eqref{eq:4thsingular} will be decoupled to
find $w\in H_0^1(\Omega)$, $\boldsymbol{\phi}\in  \boldsymbol H_0^1(\Omega; \mathbb{R}^2)$, $p\in L_0^{2}(\Omega)$ and $u\in H_0^1(\Omega)$ such that
\begin{align*}
(\curl  w, \curl  v)&=\langle f, v\rangle  \quad\quad\quad\quad \forall~v\in H_0^1(\Omega),  \\
(\boldsymbol\phi, \boldsymbol\psi)+\varepsilon^2(\boldsymbol\nabla\boldsymbol\phi, \boldsymbol\nabla\boldsymbol\psi) + (\div\boldsymbol\psi, p) & =(\curl  w, \boldsymbol\psi) \quad\;\;\, \forall~\boldsymbol{\psi}\in  \boldsymbol H_0^1(\Omega; \mathbb{R}^2), \\
(\div\boldsymbol\phi, q) &= 0 \quad\quad\quad\quad\quad\;\;\, \forall~ q\in L_0^{2}(\Omega),  \\
(\curl  u, \curl  \chi)&= (\boldsymbol{\phi}, \curl \chi) \quad\quad \forall~\chi\in H_0^1(\Omega).
\end{align*}
Note that we derived this decoupling in 2016 independently of \cite{Gallistl2017}.
The second and third equations form the Brinkman problem.
The robust wellposeness of the Brinkman problem with respect to the parameter $\varepsilon$ can be found in \cite{MardalWinther2004,OlshanskiiPetersReusken2006,MardalSchoberlWinther2012,MardalTaiWinther2002,XieXuXue2008}.
\XH{In \cite{MardalWinther2004,OlshanskiiPetersReusken2006,MardalSchoberlWinther2012}, the norms of the spaces $\boldsymbol L^2(\Omega; \mathbb{R}^2)\cap \varepsilon\boldsymbol H_0^1(\Omega; \mathbb{R}^2)$ and $H^1(\Omega)\cap L_0^{2}(\Omega)+\varepsilon^{-1}L_0^{2}(\Omega)$ were adopted for $\boldsymbol{\phi}$ and $p$ respectively, which induced an efficient and robust preconditioner.}

We then discuss the decoupling of \eqref{eq:biharmonicprimal} in three dimensions.
Based on the commutative diagram \eqref{Hdiv_1diagram}, the primal formulation \eqref{eq:biharmonicprimal} of the biharmonic equation in three dimensions is equivalent to
find $w\in H_0^1(\Omega)$, $\boldsymbol{\phi}\in  \boldsymbol H_0^1(\Omega; \mathbb{R}^3)$, $\bs p\in \bs L^{2}(\Omega; \mathbb{R}^3)/\nabla H^1(\Omega)$ and $u\in H_0^1(\Omega)$ such that (cf. \cite{Gallistl2017})
\begin{align*}
(\nabla  w, \nabla  v)&=\langle f, v\rangle  \quad\quad\; \forall~v\in H_0^1(\Omega),  \\
(\boldsymbol\nabla\boldsymbol\phi, \boldsymbol\nabla\boldsymbol\psi) + (\curl\boldsymbol\psi, \bs p) & =(\nabla w, \boldsymbol\psi) \quad \forall~\boldsymbol{\psi}\in  \boldsymbol H_0^1(\Omega; \mathbb{R}^3), \\
(\curl\boldsymbol\phi, \bs q) &= 0 \quad\quad\quad\quad \forall~ \bs q\in \bs L^{2}(\Omega; \mathbb{R}^3)/\nabla H^1(\Omega),  \\
(\nabla u, \nabla \chi)&= (\boldsymbol{\phi}, \nabla \chi) \quad\, \forall~\chi\in H_0^1(\Omega).
\end{align*}


\subsection{HHJ mixed formulation}
%
The Hellan-Herrmann-Johnson (HHJ) mixed formulation~\cite{Hellan1967,Herrmann1967,Johnson1973} of the biharmonic equation \eqref{eq:biharmonic} in two dimensions is to find $(\boldsymbol{\sigma} , u)\in \boldsymbol{H}^{-1}(\div \boldsymbol{\div },\Omega; \mathbb{S})\times H_0^1(\Omega)$ such that
\begin{align}
(\boldsymbol\sigma, \boldsymbol\tau)+ \langle \div \boldsymbol{\div }\boldsymbol\tau, u\rangle & =0 \quad\quad\;\;\, \forall~\boldsymbol\tau\in\boldsymbol{H}^{-1}(\div \boldsymbol{\div },\Omega; \mathbb{S}), \label{eq:hhjmixedformulation1}\\
\langle \div \boldsymbol{\div }\boldsymbol\sigma, v\rangle & =\langle f, v\rangle  \quad \forall~v\in H_0^1(\Omega). \label{eq:hhjmixedformulation2}
\end{align}

We recall
the known decoupling from \cite{KrendlRafetsederZulehner2016, RafetsederZulehner2017}, which fits into the framework of this paper as
follows.
We construct the following commutative diagram from \eqref{cd:hhj}
\[
\begin{array}{c}
\xymatrix@R-=1.0pc{
& \boldsymbol L^2(\Omega; \mathbb S) & \\
\boldsymbol{H}^{1}(\Omega; \mathbb{R}^2) \ar[r]^-{\sym\bs\curl} & \boldsymbol{H}^{-1}(\div \boldsymbol{\div },\Omega; \mathbb{S}) \ar[r]^-{\div \boldsymbol{\div }}\ar@{}[u]|{\bigcup}
                & H^{-1}(\Omega) \ar[r]^-{} & 0 \\
&\boldsymbol H_0^1(\Omega; \mathbb S)  \ar[u]^{\boldsymbol I} & H_0^1(\Omega)\ar[u]_{\Delta}\ar[l]_{\boldsymbol\pi}}
\end{array}.
\]

According to Theorem~\ref{thm:mixeddecouple} and Helmholtz decomposition \eqref{Hdivdiv_1symdec}, the mixed formulation~\eqref{eq:hhjmixedformulation1}-\eqref{eq:hhjmixedformulation2} can be decoupled to
find $w\in H_0^1(\Omega)$, $\boldsymbol p\in \boldsymbol{H}^{1}(\Omega; \mathbb{R}^2)/\boldsymbol{RM}^{\rot}$ and $u\in H_0^1(\Omega)$ such that
\begin{align}
(\nabla w, \nabla v)&=-\langle f, v\rangle  \quad\quad\quad\quad\quad\;\;\; \forall~v\in H_0^1(\Omega),  \label{eq:hhjmixedformdecouple1}\\
(\sym\bs\curl\boldsymbol p, \sym\bs\curl\boldsymbol q) &= - (\boldsymbol\pi w, \sym\bs\curl\boldsymbol q) \quad \forall~\boldsymbol q\in \boldsymbol{H}^{1}(\Omega; \mathbb{R}^2)/\boldsymbol{RM}^{\rot},  \label{eq:hhjmixedformdecouple2}\\
(\nabla u, \nabla \chi)&=  (\boldsymbol \sigma, \boldsymbol\pi \chi) \quad\quad\quad\quad\quad\;\; \forall~\chi\in H_0^1(\Omega), \label{eq:hhjmixedformdecouple3}
\end{align}
where $\boldsymbol \sigma = \sym\bs\curl\boldsymbol p +\boldsymbol\pi w$, and
\[
\boldsymbol{RM}^{\rot}:=\mathrm{span}\left\{\left(\begin{array}{c}
1 \\ 0
\end{array}\right), \left(\begin{array}{c}
0 \\ 1
\end{array}\right), \boldsymbol{x}\right\}.
\]
The second equation is also equivalent to the linear elasticity problem
\[
(\boldsymbol{\varepsilon}(\boldsymbol p^{\perp}), \boldsymbol{\varepsilon}(\boldsymbol q^{\perp})) = - (\boldsymbol\pi w, \boldsymbol{\varepsilon}(\boldsymbol q^{\perp})) \quad \forall~\boldsymbol q^{\perp}\in \boldsymbol{H}^{1}(\Omega; \mathbb{R}^2)/\boldsymbol{RM}
\]
where rigid motion space
\[
\boldsymbol{RM}:=\mathrm{span}\left\{\left(\begin{array}{c}
1 \\ 0
\end{array}\right), \left(\begin{array}{c}
0 \\ 1
\end{array}\right), \boldsymbol{x}^{\perp}\right\}.
\]
Such decomposition is firstly obtained in~\cite{KrendlRafetsederZulehner2016}, and has been recently generalized to the mixed boundary conditions in \cite{RafetsederZulehner2017}.
Similarly, we can also recover the decomposition of the mixed formulation \eqref{eq:hhjmixedformulation1}-\eqref{eq:hhjmixedformulation2} in three dimensions in \cite{PaulyZulehner2016}.

\subsection{The primal formulation of the fourth order $\curl$ equation}
Let $\Omega\subset\mathbb R^3$ and $\bs f\in \bs H(\div, \Omega)$ with $\div\bs f=0$. Consider the fourth order $\curl$ equation
\begin{equation}\label{eq:quadcurl}
\left\{
\begin{aligned}
&(\curl)^4\bs u=\bs f\quad\quad\quad\quad\quad\quad\;\;\text{in }\Omega, \\
&\div\bs u=0\quad\quad\quad\quad\quad\quad\quad\;\;\;\text{in }\Omega, \\
&\bs u\times\bs\nu=(\curl\bs u)\times\bs\nu=\bs 0\quad\text{on }\partial\Omega.
\end{aligned}
\right.
\end{equation}
The primal formulation of \eqref{eq:quadcurl} is to find $\bs  u\in\bs H_0^2(\curl, \Omega)$
such that
\begin{equation}\label{eq:quadcurlprimal}
(\curl\curl\bs u, \curl\curl\bs v)=(\bs f, \bs v) \quad \forall~\bs v\in \bs H_0^2(\curl, \Omega),
\end{equation}
where
\begin{align*}
\bs H_0^2(\curl, \Omega):=\{&\bs v\in\bs L^2(\Omega, \mathbb{R}^3): \curl \bs v, \curl\curl \bs v\in\bs L^2(\Omega, \mathbb{R}^3), \\
& \div\bs v=0,  \textrm{ and } \bs v\times\bs n=(\curl \bs v)\times\bs n=0 \}.
\end{align*}

Setting up the following commutative diagram from \eqref{cd:Hm1curl}
\begin{equation*}
\begin{array}{c}
\xymatrix@R-=1.0pc{
\bs H_0^{1}(\Omega; \mathbb R^3) \ar[r]^-{\Delta} & \boldsymbol{H}^{-1}(\Omega; \mathbb{R}^3) & \\
L_0^2(\Omega) \ar[r]^-{\grad} & \boldsymbol H^{-1}({\curl}, \Omega) \ar[r]^-{\curl}\ar@{}[u]|{\bigcup}
                & (K_0^c)' \ar[r]^-{} & 0 \\
&\boldsymbol H_0(\div , \Omega) \ar[u]^{I} & K_0^c \ar[u]_{\curl\curl}\ar[l]_-{\curl}}
\end{array},
\end{equation*}
the corresponding mixed formulation is to find $(\boldsymbol{\gamma} , \bs u)\in \boldsymbol{H}^{-1}(\curl , \Omega)\times K_0^c$ such that
\begin{align}
(\boldsymbol\gamma, \boldsymbol\beta)_{-1}- \langle \curl \boldsymbol\beta, \bs u\rangle & =0 \quad\quad\quad \forall~\boldsymbol\beta\in\boldsymbol{H}^{-1}(\curl ,\Omega), \label{eq:quarticcurlmixeda1}\\
\langle \curl \boldsymbol\gamma, \bs v\rangle & =(\bs f, \bs v)  \quad \forall~\bs v\in K_0^c. \label{eq:quarticcurlmixeda2}
\end{align}

By introducing variable $\boldsymbol{\phi}$ and applying the Helmholtz decomposition \eqref{eq:Hcurl_1helmholtzdecomp}
\[
\boldsymbol{\phi}=-\boldsymbol{\Delta}^{-1}\boldsymbol\gamma=-\boldsymbol{\Delta}^{-1}(\curl \bs w+\nabla p) \in \boldsymbol{H}_0^1(\Omega; \mathbb{R}^3),
\]
from Theorem~\ref{thm:mixeddecouple} the decoupled and unfolded system is to
find $\bs w\in K_0^c$, $\boldsymbol{\phi}\in  \boldsymbol H_0^1(\Omega; \mathbb{R}^3)$, $p\in L_0^{2}(\Omega)$ and $\bs u\in K_0^c$ such that
\begin{align}
(\curl  \bs w, \curl  \bs v)&=(\bs f, \bs v)  \quad\quad\quad\quad \forall~\bs v\in K_0^c,  \label{eq:quarticcurldecouple1}\\
(\boldsymbol\nabla\boldsymbol\phi, \boldsymbol\nabla\boldsymbol\psi) + (\div\boldsymbol\psi, p) & =(\curl  \bs w, \boldsymbol\psi) \quad\;\;\, \forall~\boldsymbol{\psi}\in  \boldsymbol H_0^1(\Omega; \mathbb{R}^3), \label{eq:quarticcurldecouple2}\\
(\div\boldsymbol\phi, q) &= 0 \quad\quad\quad\quad\quad\;\;\;\, \forall~ q\in L_0^{2}(\Omega),  \label{eq:quarticcurldecouple3}\\
(\curl \bs  u, \curl \bs  \chi)&= (\boldsymbol{\phi}, \curl \bs \chi) \quad\quad \forall~\bs \chi\in K_0^c. \label{eq:quarticcurldecouple4}
\end{align}
According to \eqref{eq:quarticcurldecouple3}-\eqref{eq:quarticcurldecouple4}, we have $\curl\bs u=\boldsymbol{\phi}\in\boldsymbol H_0^1(\Omega; \mathbb{R}^3)$. Note that
\[
(\boldsymbol\nabla\boldsymbol\phi, \boldsymbol\nabla\boldsymbol\psi)=(\curl\boldsymbol\phi, \curl\boldsymbol\psi)+(\div\boldsymbol\phi, \div\boldsymbol\psi).
\]
Thus we get from \eqref{eq:quarticcurldecouple2}
\[
(\curl\curl\bs u, \curl\curl\bs v)=(\curl \bs w, \curl \bs v)
\]
for any $\bs v\in K_0^c$ satisfying $\curl\bs v\in\boldsymbol H_0^1(\Omega; \mathbb{R}^3)$.
Combined with \eqref{eq:quarticcurldecouple1}, the decoupled formulation \eqref{eq:quarticcurldecouple1}-\eqref{eq:quarticcurldecouple4} is equivalent to the primal formulation \eqref{eq:quadcurlprimal} of the fourth order $\curl$ equation. The decoupling \eqref{eq:quarticcurldecouple1}-\eqref{eq:quarticcurldecouple4} is also presented in \cite{Zhang2016a} independently and based on a different approach.

Therefore we can solve the four $\curl$ problem by solving two Maxwell's equations and one Stokes equation. The Maxwell's equation with divergence-free constraint can be further decoupled into one vector Poisson equation and one scalar Poisson equation~\cite{ChenWuZhongZhou2016}.

\subsection{The mixed formulation of the fourth order $\curl$ equation}\label{sec:quadcurlmixed}
A mixed formulation of the fourth order $\curl$ equation \eqref{eq:quadcurl} is to find $\boldsymbol{\phi}\in \boldsymbol{H}(\curl\curl, (K_0^c)'$ and $\bs u\in K_0^c$ such that
\begin{align}
(\boldsymbol\phi, \boldsymbol\psi)+ \langle \curl\curl\boldsymbol\psi, \bs u\rangle & =0 \quad\quad\;\;\;\, \forall~\boldsymbol\psi\in\boldsymbol{H}(\curl\curl, (K_0^c)'), \label{eq:quadcurlmixedformulation1}\\
\langle \curl\curl\boldsymbol\phi, \bs v\rangle & =( \bs f, \bs v)  \quad \forall~\bs v\in K_0^c. \label{eq:quadcurlmixedformulation2}
\end{align}

The following commutative diagram is designed from \eqref{cd:curlcurldual}
\[
\begin{array}{c}
\xymatrix@R-=1.0pc@C=1.6cm{
& \boldsymbol L^2(\Omega; \mathbb R^3) & \\
\boldsymbol{H}^0(\bs\div, \dev\sym) \ar[r]^-{\spn^{-1}\skw} & \boldsymbol{H}(\curl\curl, (K_0^c)') \ar[r]^-{\curl\curl}\ar@{}[u]|{\bigcup}
                & (K_0^c)' \ar[r]^-{} & 0 \\
&K_0^c  \ar[u]^{I} & K_0^c\ar[u]_{\curl\curl}\ar[l]_{I}}
\end{array}.
\]
Applying Theorem~\ref{thm:mixeddecouple} and Helmholtz decomposition \eqref{complex:curlcurldual},
the mixed formulation \eqref{eq:quadcurlmixedformulation1}-\eqref{eq:quadcurlmixedformulation2} of the fourth order $\curl$ equation is equivalent to find $\bs w\in K_0^c$, $\bs \sigma\in \boldsymbol{H}^0(\bs\div, \dev\sym)$ and $\bs u\in K_0^c$ such that
\begin{align}
(\curl\bs w, \curl\bs v)&=\langle \bs f, \bs v\rangle \quad\quad\quad\quad\quad\;\; \forall~\bs v\in K_0^c, \label{eq:quadcurlmixedformdecouple1}\\
(\skw\boldsymbol\sigma, \skw\boldsymbol\tau) &= - (\spn\bs w, \skw\boldsymbol\tau)\quad \forall~\bs \tau\in \boldsymbol{H}^0(\bs\div, \dev\sym),  \label{eq:quadcurlmixedformdecouple2}\\
(\curl\bs u, \curl\bs\chi)&= -(\boldsymbol \phi, \bs\chi) \quad\quad\quad\quad\;\; \forall~\bs\chi\in K_0^c, \label{eq:quadcurlmixedformdecouple3}
\end{align}
where $\boldsymbol \phi = \spn^{-1}\skw\boldsymbol\sigma +\bs w$.

By \eqref{eq:H0divdevsym}, let $\bs\sigma=\rho\bs I+\spn\bs\alpha$. Then \eqref{eq:quadcurlmixedformdecouple2} is equivalent to find $\bs\alpha\in\boldsymbol{L}^2(\Omega; \mathbb R^3)$, $\rho\in L_0^2(\Omega)$ and $\bs p\in \bs H_0^{1}(\Omega; \mathbb R^3)$ such that
\begin{align*}
(\bs\alpha, \bs\beta) + (\gamma\bs I+\spn\bs\beta, \nabla \bs p) & =-(\bs w, \bs\beta) \quad\;\;\;\, \forall~\bs\beta\in\boldsymbol{L}^2(\Omega; \mathbb R^3), \gamma\in L_0^2(\Omega), \\
(\rho\bs I+\spn\bs\alpha, \nabla \bs q) &= 0 \quad\quad\quad\quad\quad \forall~ \bs q\in \bs H_0^{1}(\Omega; \mathbb R^3).
\end{align*}
Again both \eqref{eq:quadcurlmixedformdecouple1} and \eqref{eq:quadcurlmixedformdecouple3} are Maxwell's equations.

\subsection{The primal formulation of the triharmonic equation}\label{sec:triharmonicprimal}

Consider the triharmonic equation with homogenous Dirichlet boundary condition
\begin{equation}\label{eq:triharmonic}
\left\{
\begin{aligned}
&-\Delta^3u=f\quad\quad\quad\quad\quad\quad\text{in }\Omega,
\\
&u=\partial_{\nu}u=\partial_{\nu\nu}u=0\quad\quad\text{on }\partial\Omega,
\end{aligned}
\right.
\end{equation}
where $f\in H^{-2}(\Omega)$ with $\Omega\subset \mathbb{R}^2$.
The primal formulation of \eqref{eq:triharmonic} is to find $u\in H_0^3(\Omega)$ such that
\begin{equation}\label{eq:triharmonicprimal}
(\nabla^3u, \nabla^3 v)=\langle f, v\rangle\quad \forall~v\in H_0^3(\Omega).
\end{equation}

In view of the following commutative diagram derived from \eqref{cd:Rotrot_2symdec}
\[
\begin{array}{c}
\xymatrix@R-=1.0pc{
\bs H_0^{1}(\Omega; \mathbb S) \ar[r]^-{\Delta}  & \boldsymbol{H}^{-1}(\Omega; \mathbb{S}) & \\
\boldsymbol L^{2}(\Omega; \mathbb{R}^2) \ar[r]^-{\boldsymbol\varepsilon} & \boldsymbol{H}^{-2}(\rot \mathbf{rot},\Omega; \mathbb{S}) \ar[r]^-{\rot \mathbf{rot}}\ar@{}[u]|{\bigcup}
                & H^{-2}(\Omega) \ar[r]^-{} & 0 \\
&\boldsymbol H_0(\div , \Omega; \mathbb{S})  \ar[u]^-{\boldsymbol I} & H_0^2(\Omega)\ar[u]_{\Delta^2}\ar[l]_-{\bs\curl\curl }}
\end{array},
\]
the corresponding mixed formulation is to find $(\boldsymbol{\gamma} , u)\in \boldsymbol{H}^{-2}(\rot \mathbf{rot},\Omega; \mathbb{S})\times H_0^2(\Omega)$ such that
\begin{align}
(\boldsymbol\gamma, \boldsymbol\beta)_{-1}- \langle \rot \mathbf{rot}\boldsymbol\beta, u\rangle & =0 \quad\quad\;\;\, \forall~\boldsymbol\beta\in\boldsymbol{H}^{-2}(\rot \mathbf{rot},\Omega; \mathbb{S}), \label{eq:triharmonicmixeda1}\\
\langle \rot \mathbf{rot}\boldsymbol\gamma, v\rangle & =\langle f, v\rangle  \quad \forall~v\in H_0^2(\Omega). \label{eq:triharmonicmixeda2}
\end{align}


By introducing variable and applying the Helmholtz decomposition \eqref{Rotrot_2symdec}
\[
\boldsymbol{\phi}=-\boldsymbol{\Delta}^{-1}\boldsymbol\gamma=-\boldsymbol{\Delta}^{-1}(\boldsymbol\varepsilon\boldsymbol(p)+\bs\curl\curl  w) \in \boldsymbol{H}_0^1(\Omega; \mathbb{S}),
\]
from Theorem~\ref{thm:mixeddecouple} the decoupled system is to
find $w\in H_0^2(\Omega)$, $\boldsymbol{\phi}\in  \boldsymbol H_0^1(\Omega; \mathbb{S})$, $\boldsymbol p\in \boldsymbol L^{2}(\Omega; \mathbb{R}^2)/\boldsymbol{RM}$ and $u\in H_0^2(\Omega)$ such that
\begin{align}
(\bs\curl\curl  w, \bs\curl\curl  v)&=\langle f, v\rangle  \quad\quad\quad\quad\quad\;\;\; \forall~v\in H_0^2(\Omega),  \label{eq:triharmonicdecouple1}\\
(\boldsymbol\nabla\boldsymbol\phi, \boldsymbol\nabla\boldsymbol\psi) + (\mathbf{div}\boldsymbol\psi, \boldsymbol p) & =(\bs\curl\curl  w, \boldsymbol\psi) \quad\;\;\, \forall~\boldsymbol{\psi}\in  \boldsymbol H_0^1(\Omega; \mathbb{S}), \label{eq:triharmonicdecouple2}\\
(\mathbf{div}\boldsymbol\phi, \boldsymbol q) &= 0 \quad\quad\quad\quad\quad\quad\quad\;\; \forall~ \boldsymbol q\in \boldsymbol L^{2}(\Omega; \mathbb{R}^2)/\boldsymbol{RM},  \label{eq:triharmonicdecouple3}\\
(\bs\curl\curl  u, \bs\curl\curl  \chi)&= (\boldsymbol{\phi}, \bs\curl\curl \chi) \quad\quad \forall~\chi\in H_0^2(\Omega). \label{eq:triharmonicdecouple4}
\end{align}

It is evident that the primal formulation \eqref{eq:triharmonicprimal} of the triharmonic equation is equivalent to two biharmonic equations and one Stokes equation, c.f.~\eqref{eq:triharmonicdecouple1}-\eqref{eq:triharmonicdecouple4},
which is different from the decoupling in \cite{Gallistl2017}. These equations can be discretized directly, since there exist many finite elements for discretizing the biharmonic equations in the literature.

Recursively applying the decomposition, we can decouple the $m$-th harmonic equation $\Delta^m u=f$ with homogenous Dirichlet boundary condition, i.e., $u\in H_0^m(\Omega)$ into a sequence of Poisson and Stokes equations.

\subsection{The mixed formulation of the triharmonic equation}\label{sec:triharmonicmixed}
A mixed formulation of the triharmonic equation \eqref{eq:triharmonic} in two dimensions is to find $(\boldsymbol{\sigma} , u)\in \boldsymbol{H}^{-2}(\div^3, \Omega)\times H_0^2(\Omega)$ such that
\begin{align}
(\boldsymbol\sigma, \boldsymbol\tau)+ \langle \div\bs\div\bs\div\boldsymbol\tau, u\rangle & =0 \quad\quad\;\;\, \forall~\boldsymbol\tau\in\boldsymbol{H}^{-2}(\div^3, \Omega), \label{eq:triharmonicmixedformulation1}\\
\langle \div\bs\div\bs\div\boldsymbol\sigma, v\rangle & =\langle f, v\rangle  \quad \forall~v\in H_0^2(\Omega). \label{eq:triharmonicmixedformulation2}
\end{align}

Due to \eqref{cd:div3m2}, we construct the commutative diagram
\[
\begin{array}{c}
\xymatrix@R-=1.0pc@C=1.6cm{
& \boldsymbol{L}^2(\Omega; \mathbb S(3)) & \\
\boldsymbol{H}^1(\Omega; \mathbb S) \ar[r]^-{\sym\bs\curl} & \boldsymbol{H}^{-2}(\div^3, \Omega) \ar[r]^-{\div\bs\div\bs\div}\ar@{}[u]|{\bigcup}
                & H^{-2}(\Omega) \ar[r]^-{} & 0 \\
&H_0^1(\Omega; \mathbb R^2)  \ar[u]^{\bs\Xi} & H_0^{2}(\Omega)\ar[u]_{\Delta^2}\ar[l]_{\nabla}}
\end{array}.
\]
Then according to Theorem~\ref{thm:mixeddecouple} and the Helmholtz decomposition \eqref{complex:div3m2}, the mixed formulation \eqref{eq:triharmonicmixedformulation1}-\eqref{eq:triharmonicmixedformulation2} of the triharmonic equation is equivalent to find $w\in H_0^{2}(\Omega)$, $\bs \phi\in \boldsymbol{H}^{1}(\Omega; \mathbb{S})$ and $u\in H_0^{2}(\Omega)$ such that
\begin{align}
(\Delta w, \Delta v)&=\langle f, v\rangle  \quad\quad\quad\quad\quad\quad\quad\;\;\; \forall~v\in H_0^2(\Omega),  \label{eq:triharmonicmixedformdecouple1}\\
(\sym\bs\curl\boldsymbol\phi, \sym\bs\curl\boldsymbol\psi) &= - (\bs\Xi\nabla w, \sym\bs\curl\boldsymbol\psi) \quad \forall~\boldsymbol\psi\in \boldsymbol{H}^{1}(\Omega; \mathbb{S}),  \label{eq:triharmonicmixedformdecouple2}\\
(\Delta u, \Delta \chi)&= -(\boldsymbol \sigma, \bs\Xi\nabla\chi) \quad\quad\quad\quad\quad\, \forall~\chi\in H_0^2(\Omega), \label{eq:triharmonicmixedformdecouple3}
\end{align}
where $\boldsymbol \sigma = \sym\bs\curl\bs\phi+\bs\Xi\nabla w$.

\section{Discretization Based on Decoupled Formulation}
In this section, we will consider discretization based on the decoupled formulation.
By decoupling the fourth order equation into second order equations, we can use well known conforming finite element spaces. Furthermore, we can easily derive the superconvergence to the Galerkin projections without any mesh conditions, which is not known in the literature.

\subsection{Decoupled discretization of HHJ formulation}\label{sec:decoupledHHJ}
Let $f\in L^2(\Omega)$, $V_h\subset H_0^1(\Omega)$ and $\boldsymbol P_h\subset\boldsymbol{H}^{1}(\Omega; \mathbb{R}^2)$.
The discrete method based on formulation~\eqref{eq:hhjmixedformdecouple1}-\eqref{eq:hhjmixedformdecouple3}
is to
find $w_h\in V_h$, $\boldsymbol p_h\in \boldsymbol P_h/\boldsymbol{RM}^{\rot }$ and $u_h\in V_h$ such that
\begin{align}
(\nabla w_h, \nabla v_h)&=-( f, v_h)  \quad\quad\quad\quad\quad\quad\;\, \forall~v_h\in V_h,  \label{eq:hhjdecouplefem1}\\
(\sym\bs\curl\boldsymbol p_h, \sym\bs\curl\boldsymbol q_h) &= - (\boldsymbol\pi w_h, \sym\bs\curl\boldsymbol q_h) \quad \forall~\boldsymbol q_h\in \boldsymbol P_h/\boldsymbol{RM}^{\rot },  \label{eq:hhjdecouplefem2}\\
(\nabla u_h, \nabla \chi_h)&=  (\boldsymbol \sigma_h, \boldsymbol\pi \chi_h) \quad\quad\quad\quad\quad\;\;\forall~\chi_h\in V_h, \label{eq:hhjdecouplefem3}
\end{align}
where $\boldsymbol \sigma_h = \sym\bs\curl\bs p_h +\boldsymbol\pi w_h$.

Define projection $\bs P_h^{cs}: \boldsymbol{H}^{1}(\Omega; \mathbb{R}^2)\to\boldsymbol P_h/\boldsymbol{RM}^{\rot }$ by
\[
(\sym\bs\curl\bs P_h^{cs}\boldsymbol p, \sym\bs\curl\boldsymbol q_h)=(\sym\bs\curl\boldsymbol p, \sym\bs\curl\boldsymbol q_h).
\]
Similarly, denote by $P_h^{\grad}$ the $H^1$ orthogonal projection onto $V_h$. Let
\[
\boldsymbol \sigma_h^{\ast} := \sym\bs\curl\bs P_h^{cs}\boldsymbol p +\boldsymbol\pi w_h.
\]

\begin{lemma}\label{lem:errorestimate1}
Let $(w, \bs p, u)$ be the solution of HHJ mixed formulation \eqref{eq:hhjmixedformdecouple1}-\eqref{eq:hhjmixedformdecouple3} and  $(w_h, \bs p_h, u_h)$ be the solution of \eqref{eq:hhjdecouplefem1}-\eqref{eq:hhjdecouplefem3}. We then have the estimates
\[
|w-w_h|_1\lesssim \inf_{v_h\in V_h} |w-v_h|_1,
\]
\[
\|\sym\bs\curl(\bs P_h^{cs}\boldsymbol p-\boldsymbol p_h)\|_0 + \|\boldsymbol \sigma_h^{\ast}-\boldsymbol \sigma_h\|_0\lesssim \|w-w_h\|_0,
\]
\[
|P_h^{\grad}u-u_h|_1\lesssim \|\boldsymbol p-\boldsymbol p_h\|_0 + \|w-w_h\|_0.
\]
\end{lemma}
\begin{proof}
Subtracting \eqref{eq:hhjdecouplefem1}-\eqref{eq:hhjdecouplefem3} from \eqref{eq:hhjmixedformdecouple1}-\eqref{eq:hhjmixedformdecouple3}, we get the error equations
\begin{align*}
(\nabla(w-w_h), \nabla v_h)&=0  \quad\quad\quad\quad\quad\quad\quad\quad\quad\quad\;\;\; \forall~v_h\in V_h,  \\
(\sym\bs\curl(\bs P_h^{cs}\boldsymbol p-\boldsymbol p_h), \sym\bs\curl\boldsymbol q_h) &= (\boldsymbol\pi (w_h-w), \sym\bs\curl\boldsymbol q_h) \;\; \forall~\boldsymbol q_h\in \boldsymbol P_h/\boldsymbol{RM}^{\rot },  \\
(\nabla (P_h^{\grad}u-u_h), \nabla \chi_h)&=  (\boldsymbol \sigma-\boldsymbol \sigma_h, \boldsymbol\pi \chi_h) \quad\quad\quad\quad\quad\, \forall~\chi_h\in V_h. 
\end{align*}
Then all the error estimates hold by standard argument.
\end{proof}

Furthermore, assume
\begin{equation}\label{eq:L2H1errorrelation}
\|w-w_h\|_0\lesssim h^{\delta}|w-w_h|_1,\quad \|\boldsymbol p-\boldsymbol p_h\|_0\lesssim h^{\delta}\|\sym\bs\curl(\boldsymbol p-\boldsymbol p_h)\|_0,
\end{equation}
where $\delta \in (1/2,1]$ is the regularity constant depending on the shape of $\Omega$. This assumption can be proved by the duality argument (cf.~\cite{Ciarlet1978, BrennerScott2008}).

\begin{theorem}
Let $(w, \bs p, u)$ be the solution of HHJ mixed formulation \eqref{eq:hhjmixedformdecouple1}-\eqref{eq:hhjmixedformdecouple3} and  $(w_h, \bs p_h, u_h)$ be the solution of \eqref{eq:hhjdecouplefem1}-\eqref{eq:hhjdecouplefem3}. We then have the estimates
\[
\|\sym\bs\curl(\boldsymbol p-\boldsymbol p_h)\|_0 + \|\boldsymbol \sigma-\boldsymbol \sigma_h\|_0\lesssim \inf_{\boldsymbol q_h\in \boldsymbol P_h} \|\sym\bs\curl(\boldsymbol p-\boldsymbol q_h)\|_0+\inf_{v_h\in V_h} |w-v_h|_1,
\]
\[
|u-u_h|_1\lesssim \inf_{v_h\in V_h} |u-v_h|_1+\inf_{\boldsymbol q_h\in \boldsymbol P_h} \|\sym\bs\curl(\boldsymbol p-\boldsymbol q_h)\|_0 + \inf_{v_h\in V_h} |w-v_h|_1.
\]
Moreover if assumption \eqref{eq:L2H1errorrelation} is true, we have the improved error estimates
\begin{equation}\label{eq:superconvergentestimate1}
\|\sym\bs\curl(\bs P_h^{cs}\boldsymbol p-\boldsymbol p_h)\|_0  + \|\boldsymbol \sigma_h^{\ast}-\boldsymbol \sigma_h\|_0 \lesssim h^{\delta}\inf_{v_h\in V_h} |w-v_h|_1,
\end{equation}
\begin{equation}\label{eq:superconvergentestimate2}
|P_h^{\grad}u-u_h|_1\lesssim h^{\delta} \inf_{\boldsymbol q_h\in \boldsymbol P_h} \|\sym\bs\curl(\boldsymbol p-\boldsymbol q_h)\|_0 + h^{\delta}\inf_{v_h\in V_h} |w-v_h|_1.
\end{equation}
\end{theorem}
\begin{proof}
The first two error estimates can be derived from Lemma~\ref{lem:errorestimate1} and Poincar\'e inequality.
We can acquire \eqref{eq:superconvergentestimate1}-\eqref{eq:superconvergentestimate2} from Lemma~\ref{lem:errorestimate1} and \eqref{eq:L2H1errorrelation}.
\end{proof}

\begin{remark}\rm
The error estimates \eqref{eq:superconvergentestimate1}-\eqref{eq:superconvergentestimate2} are superconvergent if we use equal order finite element spaces for $V_h$ and $\boldsymbol P_h$. $\hfill\Box$
\end{remark}

\subsection{Decoupled discretization for biharmonic equation}
Now we discretize formulation~\eqref{eq:biharmonicdiscretedecouple1}-\eqref{eq:biharmonicdiscretedecouple4} using more general finite element spaces.

Let $f\in L^2(\Omega)$, $V_h\subset H_0^1(\Omega)$, $\boldsymbol X_h\subset\boldsymbol{H}_0^{1}(\Omega; \mathbb{R}^2)$ and $P_h\subset L_0^2(\Omega)$.
The discrete method based on formulation~\eqref{eq:biharmonicdiscretedecouple1}-\eqref{eq:biharmonicdiscretedecouple4}
is to find
$w_h, u_h\in V_h$, $\boldsymbol \phi_h\in\boldsymbol X_h$ and $p_h\in P_h$ such that
\begin{align}
(\curl w_h, \curl v_h) & =(f, v_h)  \quad\quad\quad\quad\; \forall~v_h\in V_h. \label{eq:biharmonicmfemc1}\\
(\nabla\boldsymbol \phi_h, \nabla\boldsymbol \psi_h) + (\div \boldsymbol \psi_h, p_h) & =(\curl w_h, \boldsymbol \psi_h) \quad\;\; \forall~\boldsymbol \psi_h\in\boldsymbol X_h, \label{eq:biharmonicmfemc2}\\
(\div \boldsymbol \phi_h, q_h) & =0 \quad\quad\quad\quad\quad\quad\;\; \forall~q_h\in P_h, \label{eq:biharmonicmfemc3}\\
(\curl u_h, \curl \chi_h) & =(\boldsymbol \phi_h, \curl \chi_h) \quad\;\;\; \forall~\chi_h\in V_h. \label{eq:biharmonicmfemc4}
\end{align}

We assume $(\boldsymbol X_h, P_h)$ is a stable finite element pair for Stokes equation (cf.~\cite{BoffiBrezziFortin2013, BoffiBrezziDemkowiczDuranEtAl2008}), i.e. it holds the inf-sup condtion
\begin{equation}\label{eq:stokesinfsupcondition}
\|q_h\|_0\lesssim \sup_{\boldsymbol \psi_h\in\boldsymbol \Sigma_h}\frac{(\div \boldsymbol \psi_h, q_h)}{|\boldsymbol \psi_h|_1} \quad \forall~q_h\in P_h.
\end{equation}
To analyze the discrete method~\eqref{eq:biharmonicmfemc1}-\eqref{eq:biharmonicmfemc4}, we rewrite it as a mixed finite element method
\begin{align*}
a(\boldsymbol \phi_h, u_h; \boldsymbol \psi_h, v_h) + b(\boldsymbol \psi_h, v_h; p_h, w_h)&=(f, v_h) \quad \forall~(\boldsymbol \psi_h, v_h)\in\boldsymbol X_h\times V_h, \\
b(\boldsymbol \phi_h, u_h; q_h, \chi_h)&=0 \quad\quad\quad\; \forall~(q_h, \chi_h)\in P_h\times V_h,
\end{align*}
where
\begin{align*}
a(\boldsymbol \phi_h, u_h; \boldsymbol \psi_h, v_h)&:=(\nabla\boldsymbol \phi_h, \nabla\boldsymbol \psi_h), \\
b(\boldsymbol \phi_h, u_h; q_h, \chi_h)&:=(\div \boldsymbol \phi_h, q_h)-(\boldsymbol \phi_h, \curl \chi_h)+ (\curl u_h, \curl \chi_h).
\end{align*}

\begin{lemma}
Assume the inf-sup condition~\eqref{eq:stokesinfsupcondition}, the following inf-sup condition holds
\begin{equation}\label{eq:biharmonicdecoupleinfsupcondition}
\|q_h\|_0+|\chi_h|_1\lesssim \sup_{(\boldsymbol \psi_h, v_h)\in\boldsymbol \Sigma_h\times V_h}\frac{b(\boldsymbol \psi_h, v_h; q_h, \chi_h)}{|\boldsymbol \psi_h|_1+|v_h|_1} \quad \forall~(q_h, \chi_h)\in P_h\times V_h.
\end{equation}
\end{lemma}
\begin{proof}
It is easy to see that
\[
|\chi_h|_1=\sup_{v_h\in V_h}\frac{(\curl v_h, \curl \chi_h)}{|v_h|_1}\leq \sup_{(\boldsymbol \psi_h, v_h)\in\boldsymbol \Sigma_h\times V_h}\frac{b(\boldsymbol \psi_h, v_h; q_h, \chi_h)}{|\boldsymbol \psi_h|_1+|v_h|_1}.
\]
It follows from \eqref{eq:stokesinfsupcondition} and Poincar\'e inequality
\begin{align*}
\|q_h\|_0 & \lesssim \sup_{\boldsymbol \psi_h\in\boldsymbol \Sigma_h}\frac{(\div \boldsymbol \psi_h, q_h)}{|\boldsymbol \psi_h|_1}=\sup_{\boldsymbol \psi_h\in\boldsymbol \Sigma_h}\frac{b(\boldsymbol \psi_h, 0; q_h, \chi_h)+(\boldsymbol \psi_h, \curl \chi_h)}{|\boldsymbol \psi_h|_1} \\
& \lesssim  |\chi_h|_1 + \sup_{(\boldsymbol \psi_h, v_h)\in\boldsymbol \Sigma_h\times V_h}\frac{b(\boldsymbol \psi_h, v_h; q_h, \chi_h)}{|\boldsymbol \psi_h|_1+|v_h|_1}.
\end{align*}
Therefore the inf-sup condition \eqref{eq:biharmonicdecoupleinfsupcondition} will be derived by combining the last two inequalities.
\end{proof}

\begin{theorem}
Let $(w, \boldsymbol \phi, p, u)$ be the solution of the mixed formulation~\eqref{eq:biharmonicdiscretedecouple1}-\eqref{eq:biharmonicdiscretedecouple4}, and $(w_h, \boldsymbol \phi_h, p_h, u_h)\in V_h\times\boldsymbol X_h\times P_h\times V_h$ be the solution of the discrete method~\eqref{eq:biharmonicmfemc1}-\eqref{eq:biharmonicmfemc4}. Assume both $V_h$ and $\bs X_h$ are $H^1$ conforming, the inf-sup condition~\eqref{eq:stokesinfsupcondition} holds, and the discrete spaces are consistent with respect to the mixed formulation~\eqref{eq:biharmonicdiscretedecouple1}-\eqref{eq:biharmonicdiscretedecouple4}, then
\begin{align}\label{eq:biharmonicdecoupleestimate}
&\|w-w_h\|_1 + \|\boldsymbol \phi-\boldsymbol \phi_h\|_1 + \|p-p_h\|_0 + \|u-u_h\|_1 \\
\lesssim & \inf_{\chi_h\in V_h}\|w-\chi_h\|_1 + \inf_{\boldsymbol \psi_h\in\boldsymbol X_h}\|\boldsymbol \phi-\boldsymbol \psi_h\|_1 + \inf_{q_h\in P_h}\|p-q_h\|_0 + \inf_{v_h\in V_h}\|u-v_h\|_1.
\end{align}
Moreover, if
\begin{equation}\label{eq:l2ofphi}
\|\boldsymbol \phi-\boldsymbol \phi_h\|_0\lesssim h^{\delta} \left (\|\boldsymbol \phi-\boldsymbol \phi_h\|_1 +  \inf\limits_{q_h\in P_h}\|p-q_h\|_0\right ),
\end{equation}
 then
\begin{equation}\label{eq:superconvergentestimate3}
|P_h^{\grad}u-u_h|_1\lesssim h^{\delta} \left (\inf_{\chi_h\in V_h}\|w-\chi_h\|_1 + \inf_{\boldsymbol \psi_h\in\boldsymbol X_h}\|\boldsymbol \phi-\boldsymbol \psi_h\|_1 + \inf_{q_h\in P_h}\|p-q_h\|_0\right ).
\end{equation}
\end{theorem}
\begin{proof}
For any $(\boldsymbol \psi_h, v_h)\in \Sigma_h\times V_h$ satisfying
\[
b(\boldsymbol \psi_h, v_h; q_h, \chi_h)=0\quad \forall~(q_h, \chi_h)\in P_h\times V_h,
\]
we have
\[
(\boldsymbol \psi_h, \curl \chi_h)= (\curl v_h, \curl \chi_h)\quad \forall~\chi_h\in V_h,
\]
which implies
\[
|v_h|_1\leq \|\boldsymbol \psi_h\|_0\lesssim |\boldsymbol \psi_h|_1.
\]
Thus
\[
|\boldsymbol \psi_h|_1^2+|v_h|_1^2\lesssim|\boldsymbol \psi_h|_1^2=a(\boldsymbol \psi_h, v_h; \boldsymbol \psi_h, v_h).
\]
Combining the inf-sup condition~\eqref{eq:biharmonicdecoupleinfsupcondition}, we will obtain the error estimate \eqref{eq:biharmonicdecoupleestimate} by mixed finite element method theory in~\cite{BoffiBrezziFortin2013}.
And \eqref{eq:superconvergentestimate3} can be derived using the similar argument adopted in Section \ref{sec:decoupledHHJ}.
\end{proof}

\bibliographystyle{siamplain}
\bibliography{references}
\end{document}